\documentclass[11pt]{article}
\usepackage[utf8]{inputenc}
\usepackage[pagebackref=true]{hyperref}
\usepackage{amssymb,amsmath,bm,bbm,amsthm,microtype,geometry,xcolor,tocbibind}
\usepackage{amstext,amsfonts,mathtools}
\usepackage{mathrsfs,bm,bbm,enumerate}
\usepackage{graphicx,color,tikz}
\usepackage{caption,subcaption,geometry}
\usepackage{babel}
\usepackage{dsfont}

\definecolor{DarkBlue}{rgb}{0.15,0.15,0.55}
\hypersetup{linktocpage=true,colorlinks=true,allcolors=DarkBlue}
\usepackage[capitalise]{cleveref}

\newtheorem{proposition}{Proposition}[section]
\newtheorem{theorem}[proposition]{Theorem}

\theoremstyle{definition}
\newtheorem{definition}[proposition]{Definition}

\def\One{\mathbbm{1}} 
\def\CF{{\widehat{\mathscr{P}}}}
\def\D{{\mathcal{D}}}
\def\S{{\mathcal{S}}}
\def\R{{\mathcal{R}}}
\def\Lop{\mathrm{L}} 
\def\Top{\mathrm{T}} 
\def\One{\mathbbm{1}} 
\def\C{ \mathbb{C}}

\def\N{ \mathbb{N}}
\def\R{ \mathbb{R}}

\def\drm{\mathrm{d}}

\def\Der{\mathrm{D}}

\newcommand{\lp}{\left(}
\newcommand{\rp}{\right)}
\newcommand{\I}{\leqslant}

\newcommand{\dd}{\, \mathrm{d}}
\def\s{\geqslant}
\newcommand{\leb}[2]{\mathrm{Leb}_{#1}\lp #2 \rp}
\newcommand{\m}{\mathcal}
\newcommand{\scal}[2]{\left \langle #1 , #2 \right \rangle }

\newcommand{\PR}{\mathbb{P}}

\newcommand{\abs}[1]{\left\lvert#1\right\rvert}

\newcommand{\expect}[1]{\mathbb{E}\left[ #1 \right] }

\title{The Domain of Definition of the L\'evy White Noise}
\author{Julien Fageot and Thomas Humeau}

\begin{document}
\maketitle

\begin{abstract}
{It is possible to construct L\'evy white noises as generalized random processes in the sense of Gel'fand and   Vilenkin, or as an independently scattered random measures introduced by  Rajput and   Rosinski.}
In this article, we unify those two approaches by extending the L\'evy white noise $\dot X$, defined as a generalized random process, to an independently scattered random measure. 
We are then able to give general integrability conditions for L\'evy white noises, thereby maximally enlarging their domain of definition.
Based on this connection, we provide new criteria for the practical determination of the domain of definition, including specific results for the subfamilies of Gaussian, symmetric-$\alpha$-stable, generalized Laplace, and compound Poisson white noises. 
We also apply our results to formulate a general criterion for the existence of generalized solutions of linear stochastic partial differential equations driven by a L\'evy white noise.

\end{abstract}



\section{Introduction} 

	\subsection{Two Constructions for the L\'evy White Noise}\label{subsec:3constructions}

This paper is dedicated to the study of $d$-dimensional L\'evy white noises.
These entities have been defined in at least two ways and are available in the literature as follows.

 \textit{The L\'evy white noise as a generalized random process:} 
		A generalized random process can be \emph{a priori} observed though  test functions $\varphi$ living in the space $\D(\R^d)$ of compactly supported and smooth functions. The construction of L\'evy white noises as generalized random processes is established by I.M. Gel'fand and N.Y. Vilenkin in \cite[Chapter III]{GelVil4}.
		There, the L\'evy white noise $\dot{X}$ is defined as a collection of random variables $\langle \dot{X}, \varphi \rangle$ with $\varphi \in \D(\R^d)$, 
		as presented in Section \ref{sec:GRP} below. 
		
		The distributional point of view of Gel'fand and Vilenkin offers the advantage of defining a L\'evy white noise as a random generalized function. It can then be used as the driving term of a linear stochastic partial differential equation of the form $\Lop s = \dot{X}$, hence benefiting from the theory of generalized functions.
		
 \textit{The L\'evy white noise as an independently scattered random measure:} 
		In this framework, the L\'evy white noise is specified as a random measure; that is, a collection of random variables $\langle \dot{X}, \One_A\rangle$ indexed by Borelian subsets $A \subset \R^d$ with finite Lebesgue measure. Independently scattered random measures are investigated by B.S. Rajput and J. Rosinski in \cite{Rajput1989spectral}. We provide a recap in Section \ref{sec:domain}. Note that Rajput and Rosinski {do not explicitly consider L\'evy white noises, which are specific independently scattered random measure with additional stationarity properties.} 

It can be argued that the specification of the L\'evy white noise as an independently scattered random measure is more informative than its description as a generalized random process. 
Indeed, random measures are random generalized functions, while the converse is in general not true.
However, the framework of generalized random processes is especially adapted for the study of linear stochastic partial differential equation (SPDE).
In particular, the (weak) derivative together with the integration (when it is well-defined) of a generalized random process is still a generalized random process, but this is false in general for independently scattered random measures. 
 
Beyond their differences, these two constructions are deeply connected and implicitly specify the same mathematical object, although describing it from different perspectives.
They range from more to less general, in the sense that measures generalize functions, and generalized functions generalize measures \cite[Chapter 1]{Schwartz1966distributions}. This remains valid for \textit{random} functions, measures, and generalized functions.
Each approach brings its own advantages, and it  is therefore interesting to precisely connect them in order to benefit from their strengths when applied to different contexts.

	\subsection{The Domain of Definition of the L\'evy White Noise}\label{subsec:whydomain}

Since we are interested in the study of linear SPDE driven by a L\'evy white noise, we start from a L\'evy white noise defined as a generalized random process. 
Any L\'evy white noise $\dot{X}$ can \emph{a priori} be observed  through a test function $\varphi$ living in the space of compactly supported and smooth functions.
However, it is often desirable in practice to extend the definition of the family of random variables $\langle \dot{X} , f \rangle$ to functions $f$ that are possibly neither smooth nor compactly supported. Hereafter, we provide some motivations in that direction.

\begin{itemize}

	\item \emph{Expansion of   L\'evy white noises  into orthonormal bases:} 
	Consider an orthonormal basis $(f_n)$ of $L^2(\R^d)$ and a L\'evy white noise $\dot{X}$. We would like to characterize when it is reasonable to consider the family of coefficients $\langle \dot{X} , f_n \rangle$. As a motivational example, we mention our recent works~\cite{aziznejad2018wavelet,Fageot2017multidimensional} where we use the Daubechies wavelets coefficients of a L\'evy white noise to accurately estimate its regularity. Daubechies wavelets are compactly supported, but have a limited smoothness~\cite{Daubechies1988orthonormal}. We shall see that the expansion on any Daubechies wavelet basis is possible for any L\'evy white noise. 
	More generally, we are interested in considering bases whose elements are neither compactly supported nor smooth. 

	\item \emph{Localizing the probability law of the L\'evy white noise:} 
	The domain of definition of the L\'evy white noise coincides with the domain of continuity of its characteristic functional. There are strong connections between the continuity properties of the characteristic functional and the localization of the process in appropriate function spaces~\cite{Fageot2017besov,Hida2004,Ito1984foundations}. The more we extend the domain of definition, the more we can learn about the regularity of the L\'evy white noise. 	
	
	\item \emph{Construction of generalized solutions of linear SPDEs driven by L\'evy white noise:} 
	By extending the domain of definition of the L\'evy white noise, one weakens the conditions on a differential operator $\Lop$ such that the stochastic differential equation $\Lop s = \dot{X}$ admits a solution $s$ as a generalized random process in $\D'(\R^d)$. Indeed, we have formally that, for $\varphi \in \D(\R^d)$,
	\begin{equation} \label{eq:formalrequirement}
		\langle s,  \varphi \rangle = \langle \Lop^{-1} \dot{X} , \varphi \rangle = \langle \dot{X} , (\Lop^{-1})^{*} \{\varphi\} \rangle,
	\end{equation}
	where $(\Lop^{-1})^{*}$ is the adjoint of $\Lop^{-1}$.
	We therefore see that $(\Lop^{-1})^{*} \{\varphi\} $ must  belong to the domain of definition of $\dot{X}$ in order to give a meaning to \eqref{eq:formalrequirement}.

	\item {\emph{From L\'evy white noises to L\'evy processes and L\'evy fields:}}
	A L\'evy process $X$ is a solution of the stochastic differential equation $\Der X = \dot{X}$ with boundary condition $X(0)=0$. Here, $\Der$ is the derivative operator and $\dot{X}$ is a L\'evy white noise in ambient dimension $d=1$. It is well known that, contrary to the L\'evy white noise, the L\'evy process is well-defined pointwise, with {c\`adl\`ag}\footnote{\emph{C\`adl\`ag} is the acronym (derived from French) for right continuous functions with left limit at each point.}  trajectories~\cite{Bertoin1998levy}. Formally, a L\'evy process  satisfies the relation $X(t) = \langle \dot{X}, \One_{[0,t]} \rangle$. In particular, we expect to be able to define rigorously $\langle \dot{X} , f \rangle$ for test functions of the form $f = \One_{[0,t]}$. This question was addressed in \cite{Lee2006levy} and will also be deduced from our results. Identical considerations can be made for the $d$-dimensional L\'evy sheet $X = (X_t)_{t\in \R^d}$ that generalizes L\'evy processes in higher dimensions~\cite{Adler1983representations,Dalang2015Levy,Dalang1992} and is solution of the stochastic differential equation $\dot{X} = \frac{\partial^d}{\partial _1 \cdots \partial_d} X$.
\end{itemize}

The previous examples show some benefits of extending the domain of definition of the L\'evy white noise. 
We also want to go further and identify the broadest set of test functions such that the random variable $\langle \dot{X}, f \rangle$ is well-defined.
It turns out that the framework of  independently scattered random measures  developed by Rajput and Rosinski \cite{Rajput1989spectral} is especially relevant to achieving this goal. 

	\subsection{Contributions and Outline}\label{subsec:contribution}

The primary contributions of this paper are as follows.
 We connect the definition of a L\'evy white noise as a generalized random process in $\D'(\R^d)$ with the theory of independently scattered random measures investigated by  B.S. Rajput and J. Rosinski \cite{Rajput1989spectral}. We rely on the work of those authors to identify in full generality the domain of definition of  the L\'evy white noise, along with the domain of definition with finite $p$-th moment (Section \ref{sec:domain}). 

\section{L\'evy White Noises as Generalized Random Processes} \label{sec:GRP}

The theory of generalized random processes was initiated in the 50's independently by I.M. Gel'fand~\cite{Gelfand1955generalized} and K. It\^o~\cite{Ito1954distributions}. It has the advantage of allowing the construction of a broad class of random processes, including many instances that do not admit a pointwise representation. 
Being the probabilistic adaptation of the theory of generalized functions of L. Schwartz, generalized random processes are very flexible, and many finite-dimensional results of probability theory have natural extension to this infinite-dimensional setting~\cite{bierme2017generalized,Cartier1963processus,Fernique1967lois,Ito1984foundations}.  
Generalized random processes have been used as a natural framework for CARMA random processes~\cite{Brockwell2010carma} and CARMA random fields~\cite{berger2019levydriven2,berger2019levydriven}, 
for the scaling limits of statistical models in quantum field theory~\cite{abdesselam2018towards,Abdesselam2020second,Simon1979functional}, where the continuous-domain limit fields are often too irregular to admit a pointwise representation~\cite{Camia2015planar,Chelkak2012conformal,Furlan2017tightness}.

The framework also lends itself to the construction of the $d$-dimensional Gaussian white noise, as is exploited in white noise analysis~\cite{Hida1980brownian,Hida2004}, whose goal is to provide an infinite dimensional stochastic calculus.
The extension of white noise analysis from Gaussian to L\'evy white noises have been investigated deeply~\cite{Oksendal2004white}, in particular to solve SPDEs~\cite{Lokka2004stochastic} driven by L\'evy white noise. However, these works mostly deal with second-order L\'evy white noises (or even L\'evy white noises with finite exponential moments), a restriction that we want to avoid. 
More generally, one can describe $d$-dimensional L\'evy white noises---including   Gaussian ones---as   random elements in the space of generalized functions  \cite{GelVil4}.  We can then study linear stochastic partial differential equations driven by a L\'evy white noise, whose solutions are defined as generalized random processes~\cite{aziznejad2018wavelet,Dalang2015Levy,dalang2019random,Fageot2017nterm,hummel2019stochastic,hummel2020sample}. 
This framework has also been applied in signal processing in order to specify stochastic models for sparse signals~\cite{Bostan.etal2013,Unser2014sparse} and images~\cite{Bostan2013map,clarkson2016characteristic,Fageot2015wavelet}. 

	\subsection{Generalized Random Processes}

We denote by $\D(\R^d)$ the   space of infinitely differentiable and compactly supported functions, associated with its usual nuclear topology  \cite{Treves1967}. In particular, $\varphi_n \rightarrow \varphi$ in $\D(\R^d)$ as $n \rightarrow \infty$ if there exists a compact $K \subset \R^d$ that contains the support of all the $\varphi_n$  and if $\lVert \Der^{\bm{\alpha}} \{ \varphi_n - \varphi \} \rVert_{\infty} \rightarrow \infty$ for every multi-index $\bm{\alpha} \in \N^d$. The topological dual of $\D(\R^d)$ is the  space of generalized functions  $\D'(\R^d)$ (often called distributions).
We define a structure of measurability on $\D'(\R^d)$ by considering the cylindrical $\sigma$-field $\mathcal{C}$; that is, the $\sigma$-field generated by the cylinders
$\{ u \in \D'(\R^d), \quad \left( \langle u,\varphi_1 \rangle, \ldots, \langle u, \varphi_N\rangle \right) \in B \}$
with $N\geq 1$, $\varphi_1, \ldots, \varphi_N \in \D(\R^d)$, and $B$ a Borel  measurable set of $\R^N$. 

Let $(\Omega, \mathcal{F}, \mathscr{P})$ be a complete probability space. A random variable is a measurable function from $\Omega$ to $\R$. The space of random variables, $L^0(\Omega)$, is a Fr\'echet space when endowed with the convergence in probability. We also define $L^p(\Omega)$, the space of random variables with finite $p$th-moment, for $0<p< \infty$. The space $L^p(\Omega)$ is a quasi-Banach space when $0<p<1$, and a Banach space for $p \geq 1$.

\begin{definition} \label{def:GRP}
	A \emph{generalized random process} is a linear and continuous function $s$ from $\D(\R^d)$ to $L^0(\Omega)$. The linearity means that, for every $\varphi_1, \varphi_2 \in \D(\R^d)$ and $\lambda \in \R$,
	\begin{equation*}
		s(\varphi_1 + \lambda \varphi_2) = s(\varphi_1) + \lambda s(\varphi_2) \text{ almost surely}.
	\end{equation*}
	The continuity means that if $\varphi_n \rightarrow \varphi$ in $\D(\R^d)$, then $s(\varphi_n)_{n\in \N}$ converges to $s(\varphi)$ in probability.
\end{definition}

Due to the nuclear structure on $\D(\R^d)$, a generalized random process has a version that is a measurable function from $(\Omega,\mathcal{F})$ to $(\D'(\R^d) , \mathcal{C})$ (see \cite[Corollary 4.2]{Walsh1986introduction}).
In other words, a generalized random process is a random generalized function.
We therefore  write $s(\varphi) = \langle s,\varphi\rangle$ where $s$ is a generalized random process and $\varphi \in \D(\R^d)$ a test function. 

\begin{definition}
Let $s$ be a generalized random process.
The \emph{probability law} of  $s$ is the probability measure on $\D'(\R^d)$ defined by 
\begin{equation*}
	\mathscr{P}_s (B) = \mathscr{P} ( \{ s \in B  \} ) = \mathscr{P} ( \{\omega \in \Omega, \ \scal s {\cdot}(\omega) \in B \} )
\end{equation*}
for $B$ in the cylindrical $\sigma$-field $\mathcal{C}$. 
{The \emph{characteristic functional} of $s$ is the Fourier transform of its probability law; that is, the functional $\CF_s : \D(\R^d) \rightarrow \C$ defined by}
\begin{equation*}
	\CF_s(\varphi) = \int_{\D'(\R^d)} \mathrm{e}^{\mathrm{i} \langle u, \varphi \rangle } \drm \mathscr{P}_s(u) = \mathbb{E} \left[ \mathrm{e}^{\mathrm{i} \langle s ,\varphi \rangle } \right].
\end{equation*}
\end{definition}

\textit{Remarks.} Note that $\{ s\in B\} \subset \Omega$ is  measurable because $s :\Omega \rightarrow \mathbb{R}$ is measurable by definition.
 The characteristic functional characterizes the law of $s$ in the sense that two random processes are equal in law if and only if they have the same characteristic functional. 

	\subsection{L\'evy White Noises, Characteristic Exponents, and L\'evy Triplets}

One of the advantages of the theory of generalized random processes is that it allows to properly define random processes which do not admit a pointwise representation. The typical example is the Gaussian white noise. Following \cite{GelVil4}, we define the L\'evy white noise from its characteristic functional.

A random variable $Y$ is said to be infinitely divisible if it can be decomposed as $Y= Y_1 + \cdots + Y_N$, {with $(Y_1, \ldots , Y_N)$ an i.i.d. random vector for every $N$}. The characteristic function of an infinitely divisible random variable cannot vanish and there exists a unique continuous function $\psi$ such that $\CF_Y(\xi) = \exp ( \psi (\xi) )$ (see \cite[Theorem 8.1]{Sato1994levy}). The  log-characteristic function $\psi$ of an infinitely divisible random variable is called its \emph{characteristic exponent}.
{The complete family of characteristic exponents is specified by the L\'evy--Khintchine decomposition. According to \cite[Theorem 8.1]{Sato1994levy}, a function $\psi : \R \rightarrow \C$ is a characteristic exponent if and only if it can  be written as
\begin{equation} \label{eq:LK}
	\psi(\xi) = \mathrm{i} \gamma \xi - \frac{\sigma^2 \xi^2}{2} + \int_{\R} \left( \mathrm{e}^{\mathrm{i} x \xi} - 1 - \mathrm{i} x \xi  \One_{\abs{x}\leq 1} \right) \nu(\drm x),
\end{equation}
for every $\xi\in \R$, with $\gamma \in \R$, $\sigma^2 \geq 0$, and $\nu$ a L\'evy measure, that is a measure on $\R$ with $\int_{\R}  \inf (1, x^2) \nu(\drm x) < \infty$ and $\nu(\{0\}) = 0$.   The triplet $(\gamma,\sigma^2, \nu)$ is uniquely determined and called the \emph{L\'evy triplet} associated to~$\psi$.}

\begin{definition}\label{levynoise}
	A \emph{L\'evy white noise} $\dot{X}$ is a generalized random process with characteristic functional of the form 
	\begin{equation*}
		\CF_{\dot{X}}(\varphi) = \exp\left( \int_{\R^d} \psi(\varphi(t)) \drm t \right)
	\end{equation*}
	for every $\varphi \in \D(\R^d)$, where $\psi$ is a characteristic exponent. 
\end{definition}
The existence of a L\'evy white noises as generalized random processes is proved in \cite{GelVil4} (see also \cite[Theorem 2]{Fageot2014}).
A L\'evy white noise $\dot{X}$ is stationary, meaning that $\dot{X}(\cdot - t_0)$ and $\dot{X}$ has the same law for every $t_0\in\R^d$.
Moreover, $\dot{X}$ is independent at every point in the sense that $\langle \dot{X},\varphi_1 \rangle$ and $\langle \dot{X} ,\varphi_2 \rangle$ are independent when $\varphi_1$ and $\varphi_2$ have disjoint supports.

The random variable $\langle \dot{X},\varphi \rangle$ is  \emph{a priori} well-defined  for $\varphi \in \D(\R^d)$. 
Its characteristic function $\Phi_{\langle s, \varphi \rangle} : \R \rightarrow \C$ is given, for every $\xi \in \R$, by
\begin{equation} \label{eq:CFsphi}
	\Phi_{\langle s, \varphi\rangle} (\xi) =\mathbb{E} \left[ \mathrm{e}^{\mathrm{i} \xi \langle \dot{X} ,\varphi \rangle } \right]
				= \exp \left( \int_{\R^d} \psi( \xi \varphi(t) ) \drm t \right).
	\end{equation}
However, we can reasonably extend the domain of definition of the noise to broader classes of functions.
We illustrate this idea on the Gaussian white noise $ \dot{X}_{\mathrm{Gauss}}$. 
Its characteristic functional is $\CF_{ \dot{X}_{\mathrm{Gauss}}}(\varphi) = \exp( - \sigma^2 \lVert \varphi \rVert^2_2 /2)$ (see \cite{GelVil4}). 
For each $\varphi \in \D(\R^d)$,   $\langle  \dot{X}_{\mathrm{Gauss}} , \varphi \rangle$ is therefore a centered normal random variable with variance $\sigma^2 \lVert \varphi \rVert_2^2$.
One sees  easily that $\langle  \dot{X}_{\mathrm{Gauss}}, f \rangle$ can be extended to every function $f \in L^2(\R^d)$. To do so, consider a sequence $(\varphi_n)_{n\in \N}$ of functions in $\D(\R^d)$ converging to $f \in L^2(\R^d)$ for the usual Hilbert topology of $L^2(\R^d)$. Then, for every $n,m \in \N$, we have 
\begin{equation} \label{eq:cauchysequencegaussian}
\mathbb{E} [ (\langle  \dot{X}_{\mathrm{Gauss}}, \varphi_n \rangle - \langle  \dot{X}_{\mathrm{Gauss}}, \varphi_m \rangle)^2 ]  =\mathbb{E} [\langle  \dot{X}_{\mathrm{Gauss}}, \varphi_n - \varphi_m \rangle^2 ] = \sigma^2 \lVert \varphi_n - \varphi_m \rVert^2_2.
\end{equation}
The sequence $(\varphi_n)_{n\in \N}$ being convergent, it is  a Cauchy sequence of $L^2(\R^d)$. Then, \eqref{eq:cauchysequencegaussian} implies that $(\langle  \dot{X}_{\mathrm{Gauss}}, \varphi_n\rangle )$ is itself a Cauchy sequence in the complete space $L^2(\Omega)$, and hence is convergent in this space. One readily shows that the limit does not depend on the sequence $(\varphi_n)_{n\in \N}$ and we denote it by $\langle  \dot{X}_{\mathrm{Gauss}}, f \rangle$. Then, as will be made more rigorous in the sequel, the linear and continuous functional $\dot{X}_{\mathrm{Gauss}}$ initially defined from $\D(\R^d)$ to $L^0(\Omega)$ is actually a linear and continuous functional from $L^2(\R^d)$ to $L^2(\Omega)$. 

Defining a Gaussian white noise only for smooth and rapidly decaying test functions appears highly conservative. As we shall see, this occurs for any L\'evy white noise. The goal of this paper is precisely to identify the domain of definition  for general L\'evy white noise; that is, the largest possible space of test functions which can be applied to the white noise. The identification of the domain has already been done in \cite{Rajput1989spectral} in the context of independently scattered random measures, and the unification of L\'evy white noise with this notion in Section \ref{unification} allows us to rely on this work. Moreover, we are aiming at practical criteria to identify this domain of definition for specific L\'evy white noises.

\section{Integrability Conditions with respect to L\'evy White Noises} \label{sec:domain}

In this section, a L\'evy white noise is understood as a generalized random process in the sense of Gel'fand and Vilenkin. In section \ref{unification}, we connect this construction with the framework of independently scattered random measures of Rajput and Rosinski \cite{Rajput1989spectral}. Then, in Section \ref{subsec:extension}   we use this connection and the results of \cite{Rajput1989spectral} to extend the domain of definition of L\'evy white noises.

	\subsection{L\'evy White Noises as Independently Scattered Random Measures}\label{unification}

A random measure is a random process whose test functions are indicator functions.
This concept is very popular for stochastic integration, the integral being defined for simple functions (\emph{i.e.}, linear combinations of indicator functions), and then extended by a limit argument.
Essentially, a random measure is independently scattered when two indicator functions with disjoint supports define independent random variables. 
For a proper definition, see \cite[Section 1]{Rajput1989spectral}.

We show in this section that a L\'evy white noise is an example of   independently scattered random measures. 
We achieve this goal by defining the random variables $\langle \dot{X}, 1_{A} \rangle$ for any Borelian set $A$ whose support has finite Lebesgue measure as limits in probability of random variables $\langle \dot{X}, \varphi\rangle$ with $\varphi \in \D(\R^d)$.

{The set of Borel measurable subsets of $\R^d$ is $\mathcal{B}(\R^d)$. We denote by $\mathcal{B}_b(\R^d)$ the set of bounded Borel measurable sets and by $\mathcal{B}_f(\R^d)$ {the set of Borel measurable sets $A$ with finite Borel measures $\leb d A < \infty$}. We clearly have the strict inclusions $\mathcal{B}_b(\R^d) \subset \mathcal{B}_f(\R^d) \subset \mathcal{B}(\R^d)$. A function $\theta \in \D(\R^d)$ is a mollifier if $\theta \geq 0$ and $\int_{\R^d} \theta(t) \mathrm{d} t = 1$. We then set $\theta_n(t)  = n
^d \theta(nt)$, the sequence $(\theta_n)_{n \geq 1}$ being an approximation of the Dirac impulse $\delta$. 

\begin{proposition} \label{prop:newdefinitionofRV}
Let $\dot{X}$ be a Lévy white noise, $A \in \mathcal{B}_b(\R^d)$, and $B \in \mathcal{B}_f(\R^d)$. Then
\begin{itemize}
    \item Let $\theta \in \D(\R^d)$ be a mollifier. Then, the sequence $(\langle \dot{X},  \theta_n * \mathds{1}_A \rangle)_{n\geq 1}$ converges in probability to a random variable $\langle \dot{X} , \mathds{1}_A \rangle$ that does not depend on the choice of the mollifier $\theta$.
    \item Let $(K_n)_{n\geq 1}$, be a increasing sequence of compact sets such that $\cup_{n\geq 1} K_n = \R^d$. Then, the sequence $(\langle \dot{X},1_{B \cap K_n} \rangle)_{n\geq 1}$ converges in probability to a random variable $\langle \dot{X} , 1_B \rangle$ that does not depend on the choice of $(K_n)_{n\geq1}$.
\end{itemize}
 In addition, the characteristic function of $ \scal{\dot X}{\mathds 1_B}$ is given by 
 \begin{equation} \label{eq:cfwindb}
     \Phi_{\scal{\dot X}{ \mathds 1_B}}(\xi)
= \exp(\leb d B \psi(\xi))
 \end{equation}
for all $\xi\in \R$, where $\psi$ is the characteristic exponent of $\dot X$. 
\end{proposition}

\textit{Remark.} Note that $\theta_n * \mathds{1}_A \in \mathcal{D}(\R^d)$ and $B\cap K_n \in \mathcal{B}_b(\R^d)$, hence the random variables $\langle \dot{X}, \theta_n * \mathds{1}_A \rangle$ and $\langle \dot{X},\mathds{1}_{B \cap K_n} \rangle$ are well-defined, the first one because $\dot{X}$ is a generalized random process in $\D'(\R^d)$, and the second due to the first point of Proposition \ref{prop:newdefinitionofRV}.

\begin{proof}
\textit{The limit of $(\langle \dot{X},  \theta_n * \mathds{1}_A \rangle)_{n\geq 1}$.}
For $n \geq 1$, we set $Y_n = \scal{\dot X}{  \theta_n * \mathds 1_A }$. The goal is to show that the sequence $(Y_n)_{n \geq 1}$ is Cauchy in probability; that is, $\mathscr{P}(|Y_n - Y_m| > \epsilon) \rightarrow 0$ when $n,m \rightarrow \infty$ for any $\epsilon > 0$. The convergence in probability to a constant is equivalent to the convergence in law to the same constant~\cite[p. 27]{billingsley2013convergence}. Applying this fact to $(Y_n - Y_m)_{n,m \geq 1}$, we deduce that it suffices to show that $\log \mathbb{E} [\mathrm{e}^{\mathrm{i} (Y_n - Y_m) \xi} ] \rightarrow 0$ when $n,m \rightarrow \infty$ for any $\xi \in \R$. 
According to \eqref{eq:CFsphi}, we have that 
\begin{equation}
    \log \mathbb{E} \left[\mathrm{e}^{\mathrm{i} (Y_n - Y_m) \xi} \right] 
=
\int_{\R^d} \psi \big( \xi   \left( \lp \theta_n - \theta_m \rp * \mathds 1_A \right) (t)  \big)  \dd t
\end{equation}
where $\psi$ can be decomposed, thanks to \eqref{eq:LK}, as
$$\psi(\xi)= \left( \mathrm{i} \gamma \xi-\frac{\sigma^2 \xi^2} 2  \right) + \int_{|x|\I1}\lp e^{\mathrm{i} \xi x} -1-\mathrm{i} \xi x \rp \nu(\! \dd x) + \int_{|x|>1}\lp e^{\mathrm{i} \xi x} -1 \rp \nu(\! \dd x) = \psi_1(\xi) + \psi_2(\xi) + \psi_3(\xi) .$$
We treat each of the three terms of the characteristic exponent separately. The support of $f*g$ is included in the sum of the supports of $f$ and $g$, $A$ is bounded, and all the $\theta_n$ have their supports included in a common compact set. Hence, there is a compact $K$ such that $\text{supp} \ (\theta_n * A) \subset K$.
Then, 
\begin{equation} \label{eq:firstmorceau}
   \left\lvert \int_{\R^d} \psi_1 \big( \xi   \left( \lp \theta_n - \theta_m \rp * \mathds 1_A \right) (t)  \big)  \dd t \right\rvert
    \leq |\gamma \xi | \lVert  ( \theta_n - \theta_m )* \mathds{1}_A \rVert_{L_1(K)} + \frac{\sigma^2\xi^2}{2} \lVert  ( \theta_n - \theta_m )* \mathds{1}_A \rVert_{L_2(K)}^2.
\end{equation}
It is well known that for $p\s 1$ and $f\in L^p(K)$, $\lp \theta_n-\theta_m\rp*f\to 0$ in $L^p(K)$ as $n,m\to \infty$. Therefore, the right term in \eqref{eq:firstmorceau} goes to $0$ when $n,m \rightarrow \infty$. 
Then, using that $\left | e^{\mathrm{i} \xi x}-1- \mathrm{i} \xi x \right | \I \frac 1 2 |\xi x|^2$ (see for instance \cite[Lemma 5.14]{Kallenberg2006foundations}), we deduce that 
\begin{equation}
    \label{eq:secondmorceau}
      \left\lvert \int_{\R^d} \psi_2 \big( \xi   \left( \lp \theta_n - \theta_m \rp * \mathds 1_A \right) (t)  \big)  \dd t \right\rvert
    \leq
    \frac{\xi^2}{2} \left( \int_{|x|\leq 1}x^2\nu(\dd x) \right) \lVert  ( \theta_n - \theta_m )* \mathds{1}_A \rVert_{L_2(K)}^2,
\end{equation}
which also vanishes when $n,m \rightarrow \infty$. 
The last term represents the compound Poisson part of the L\'evy-It\^o decomposition of the L\'evy white noise. It corresponds to the characteristic function of the random variable $M_{n,m}:=\int_{\R^d} \int_\R x  ( (\theta_n-\theta_m)*\mathds 1_A ) (t) J(\! \dd t, \dd x)$ where $J$ is a Poisson random measure on $\R^d\times \R$ with intensity measure $\! \dd t \mathds 1_{|x|>1} \nu(\! \dd x)$, and we know that 
$$M_{n,m}= \sum_{i\s 1} Z_i    ( \lp \theta_n - \theta_m \rp * \mathds 1_A ) (T_i)\, ,$$
for some random space-time points $(Z_i, T_i)_{i\s 1}$, and the sum above has finitely many terms (independently of $m,n$) almost surely due to the existence of the compact $K$ introduced above. 
Indeed, we have
$$\expect{ J\lp K\times \R \rp } = \int_{K\times \R } \dd t \mathds 1_{|x|>1}\nu(\! \dd x)=\leb d K \int_{|x|>1} \nu(\! \dd x)<\infty\, ,$$
 and $J\lp K\times \R \rp$ is the random variable  that counts the number of points $T_i$ that fall in $K$. By Lebesgue's differentiation theorem (see \cite[Chapter 7, Exercise 2]{zygmund}), $\lp\theta_n - \theta_m \rp * \mathds 1_A(t) \to 0$ as $n,m\to \infty$ for all $t\in K\backslash H$, where $H$ is a subset of $\R^d$ such that $\leb d H=0$. The random times $T_i$ have an absolutely continuous law. Indeed, for any Borel set $B\subset \R^d$,  
$$ \PR(T_i \in B ) \I \PR(J(B \times \R) \s 1)\I \leb d B \int_{|z| >1} \nu(\! \dd z)\, .$$ 
Therefore,  for all $i\s 1$, $\PR(T_i\in H)=0$ and  $M_{n,m} \to 0$ as $n,m \to \infty$ almost surely, hence also in law, which implies that $\int_{\R^d} \psi_3 \big( \xi   \left( \lp \theta_n - \theta_m \rp * \mathds 1_A \right) (t)  \big)  \dd t \rightarrow 0$ when $n,m \rightarrow 0$. Finally, we have shown that $\log \mathbb{E} \left[\mathrm{e}^{\mathrm{i} (Y_n - Y_m) \xi} \right] \rightarrow 0$ for any $\xi$ when $n,m \rightarrow \infty$ and $(Y_n)_{n\geq1}$ converges in probability to some random variable $Y = \scal{\dot{X}}{\theta_n * \mathds{1}_A}$. 
By observing the convergence of each term of the decomposition of the characteristic exponent, it is moreover easy to see that for all $\xi \in \R$,
\begin{equation} \label{eq:Ineedthis}
    \int_{\R^d} \psi \lp \xi ( \theta_n  * \mathds 1_A ) (t)\rp  \dd t \to \int_{\R^d} \psi \lp \xi \mathds{1}_A (t)\rp  \dd t = \leb d A \psi(\xi) \, , \qquad \text{as} \ n\to \infty\, ,
\end{equation}
hence $\expect{ e^{\mathrm{i} \xi Y}} =\exp \lp \leb d A \psi(\xi)  \rp$.
If $\tilde \theta$ is another mollifier, and $(\tilde Y_n)_{n\geq 1}$ and $\tilde Y$ are  the associated sequence and limit, it is easy to see by linearity of $\dot X$ that $Y_n -\tilde Y_n \to 0$ in probability as $n\to \infty$, hence the limit $Y = \tilde{Y}$ does not depend on the choice of $\theta$. \\

\textit{The limit of $(\langle \dot{X},\mathds{1}_{B \cap K_n} \rangle)_{n\geq 1}$.}
Applying the same principle as above, it is easy to show that $  (\langle \dot{X},\mathds{1}_{B \cap K_n} \rangle)_{n\geq 1}$ is Cauchy in probability, and therefore admits a limit that we denote by $\scal{\dot{X}}{\mathds{1}_B}$. Again, one show that the limit does not depend on the choice of the compact sets $K_n$ by showing that $Z_n - \tilde{Z}_n \rightarrow 0$ in probability as $n \rightarrow \infty$, where $\tilde{Z}_n$ is associated to a second sequence of increasing compact sets $\tilde{K}_n$. Finally, as we did for $A \in \mathcal{B}_f(\R^d)$ in \eqref{eq:Ineedthis}, one proves that the characteristic function of $\scal{\dot{X}}{\mathds{1}_B}$ has the expected form for $B \in \mathcal{B}_f(\R^d)$ by a limiting argument.
\end{proof}}

{The set $\mathcal{B}_f(\R^d)$ of Borelian sets of finite Lebesgue measure is a $\delta$-ring\footnote{A $\delta$-ring is a collection of sets that is closed under finite union, countable intersection, and relative complementation \cite[Definition 1.2.13]{Bogachev2007measure}. It appears in measure theory, especially when one wants to avoid sets with infinite measure.} Since $\bigcup_{n\in \N} [-n,n]^d=\R^d$, condition $(1.4)$ of \cite{Rajput1989spectral} is satisfied.} By Proposition~\ref{prop:newdefinitionofRV}, we have defined the random set function as an extension of the L\'evy white noise $\dot X$ on $\mathcal{B}_f(\R^d)$. We still refer to this set function as a L\'evy white noise.

\begin{theorem} \label{theo:connection}
 The extension of the L\'evy white noise $\dot X$ is an independently scattered random measure in the sense of \cite{Rajput1989spectral}.
\end{theorem}

\begin{proof}
{Let $A , B \in \mathcal{B}_f(\R^d)$ be two disjoint sets, so that $\mathds{1}_A + \mathds{1}_B = \mathds{1}_{A\cup B}$. Then, the generalized random process $\dot{X}$ satisfies, any $n \geq 1$, $\scal{\dot{X}}{\theta_n * \mathds{1}_{A\cup B}} = \scal{\dot{X}}{\theta_n * \mathds{1}_{A}} + \scal{\dot{X}}{\theta_n * \mathds{1}_{B}}$,  where $\theta =n^d \theta(n \cdot)$ for some mollifier $\theta$. Due to the convergence in the first part of Proposition \ref{prop:newdefinitionofRV}, we deduce that $\scal{\dot{X}}{\mathds{1}_{A\cup B}} = \scal{\dot{X}}{\mathds{1}_{A}} + \scal{\dot{X}}{\mathds{1}_{B}}$.  The same conclusions hold more generally when summing finitely many random variables coming from disjoint sets in $\mathcal{B}_f(\R^d)$.}

For the definition of a scattered random measure, we refer to~\cite[p.455]{Rajput1989spectral}. Let $(A_n)_{n\in \N}$ be a sequence of disjoint sets in $\mathcal{B}_f(\R^d)$. Let $k\in \N$ and $i_1<\dots < i_k \in \N$. We show that the random variables $\scal{\dot X}{ \mathds 1_{A_{i_j}}}$,  $1\I j \I k$ are independent. By linearity of the noise, this fact is an immediate consequence of \eqref{eq:cfwindb} and the $\sigma$-additivity of Lebesgue measure. This proves that $\scal{\dot X}{ \mathds 1_{A_{n}}}$,  $n\in \N$ is a sequence of independent random variables. If in addition $\sum_{n\in \N} \leb d {A_n} <\infty$,  then $\bigcup_{n\in \N} A_n \in \mathcal{B}_f(\R^d)$ and we need to show that 
$$\scal{\dot X}{ \mathds 1_{\bigcup_{n\in \N} A_{n}}}= \sum_{n\in \N} \scal{\dot X}{ \mathds 1_{A_{n}}}\, ,$$
 where the series converges almost surely. As we have seen, we have for any $k\in \N$, 
\begin{equation}\label{sum}
 \scal{\dot X}{ \mathds 1_{\bigcup_{n=1}^k A_{n}}}= \sum_{n=1}^k \scal{\dot X}{ \mathds 1_{A_{n}}}.
\end{equation}
From the expression of the characteristic function of the left-hand side of \eqref{sum} and by linearity, we see that $\scal{\dot X}{ \mathds 1_{\bigcup_{n=1}^k A_{n}}} \to \scal{\dot X}{ \mathds 1_{\bigcup_{n\in \N} A_{n}}}$ in probability as $k\to \infty$. Therefore the right-hand side of \eqref{sum} is a sum of independent random variables that converges in probability. By \cite[Theorem 5.3.4]{chung}, the sum converges almost surely, which concludes the proof.
\end{proof}

	\subsection{Extension of the Domain of Definition of the Lévy White Noise} \label{subsec:extension}
Having connected L\'evy white noises in $\m D'(\R^d)$ with independently scattered random measures, it is then possible to construct a stochastic integral of non-random functions. This is done in \cite{Rajput1989spectral} and we simply restate the main definitions and theorems for the convenience of the reader. For any simple function $f=\sum_{i=1}^n a_i \mathds 1_{A_i}$ where $\leb d {A_i}<\infty$ for any $1\I i\I n$, we can define for any Borel set $A$,
\begin{equation*}
 \scal{\dot X}{f\mathds 1_A}:=\sum_{i=1}^n a_i \scal{\dot X}{\mathds 1_{A_i\cap A}}\, .
\end{equation*}
\begin{definition}\label{defrosinski}
 We say that a Borel-measurable function $f: \R^d \to \R$ is \emph{$\dot X$-integrable} if there exists a sequence of simple function $(f_n)_{n\in \N}$ such that $f_n\to f$ almost everywhere for the Lebesgue measure as $n\to \infty$, and for any Borel set $A$, the sequence $\lp \scal{\dot X}{f_n \mathds 1_A}\rp_{n\in \N}$ converges in probability as $n\to\infty$. In this case we define for any Borel set $A$,
\begin{equation*}
 \scal{\dot X}{f \mathds 1_A}:=\lim_{n\to\infty} \scal{\dot X}{f_n \mathds 1_A}\, .
\end{equation*}
{We denote by $L^0(\dot X)$ the set of all $\dot X$-integrable functions. For $p > 0$, we also define the space of $\dot X$-integrable functions that have a finite $p$th moment; that is,}
	\begin{equation*}
	{L^p (\dot{X}) = \{ f \in L^0(\dot{X}), \ \mathbb{E}[\lvert \langle \dot{X}, f \rangle \rvert^p] < \infty \} \subset L^0(\dot{X}).}
	\end{equation*}
\end{definition}

Definition \ref{defrosinski} identifies the class of measurable test functions $f$ such that $\langle \dot{X}, f \rangle$ is well-defined {and the class of measurable test functions $f$ such that $\mathbb{E}[\lvert \langle \dot{X}, f \rangle \rvert^p] < \infty$.
For $p \geq 0$, we define the \emph{$p$th-order Rajput--Rosinski exponent}
\begin{align} \label{eq:Psi}
\Psi_p (\xi) =   \Big\lvert \gamma \xi + \int_{\R} x \xi \left( \One_{\abs{x \xi}\leq 1} - \One_{\abs{x}\leq 1} \right) \nu(\drm x) \Big\rvert  
				+ \sigma^2 \xi^2
				+ \int_{|x \xi|\leq 1}|x|^p |\xi|^p  \nu(\drm x)
				+ \int_{|x \xi|> 1} x^2\xi^2  \nu(\drm x).
\end{align}
For $p = 0$, we may use the notation $\Psi_p = \Psi$, that we simply call the \emph{Rajput--Rosinski exponent} of $\dot{X}$. The Rajput--Rosinski exponent allows us to reformulate the main results of \cite{Rajput1989spectral} in Proposition~\ref{integrable}.

\begin{proposition}\label{integrable}
 Let $ 0 \leq p < \infty$ and $\dot X$ be a L\'evy white noise with 
$p$th order Rajput--Rosinski exponent $\Psi_p$. 
Then $f \in L^p(\dot X)$ if and only if $\Psi(f) := \int_{\R^d} \Psi(f(t)) \drm t < \infty$. 
\end{proposition}

The case $p=0$ is covered in~\cite[Theorem 2.7]{Rajput1989spectral}, while the case $p>0$ is a direct consequence of \cite[Theorem 3.3]{Rajput1989spectral}. 
It turns out that the spaces $L^p(\dot{X})$ have a rich structure of generalized Orlicz spaces.}
The definition and first properties of those spaces are recalled in Appendix \ref{app:orlicz}.
We refer to \cite{Rao1991theory} for an in-depth exposition, with a special emphasis on the Chapter X. 

{\begin{proposition} \label{prop:LdotXOrlicz}
	Let $ 0 \leq p < \infty$. The $p$th-order Rajput--Rosinski exponent $\Psi_p$ of a L\'evy white noise $\dot{X}$ is a $\Delta_2$-regular $\varphi$-function in the sense of Definition \ref{def:phifun} (see Appendix \ref{app:orlicz}).
	Therefore, $L^p(\dot{X}) = L^{\Psi_p}(\R^d)$ is a generalized Orlicz space. It is therefore a complete linear metric space equipped with the F-norm
	\begin{equation*}
		\lVert f \rVert_{\Psi_p} := \inf \{ \lambda > 0, \Psi_p( f / \lambda ) \leq \lambda \}.
	\end{equation*}
	The space $\D(\R^d)$ is dense in $L^p(\dot{X})$ and the convergence of a sequence of functions $f_k$ in $L^p(\dot{X})$ to $0$ is equivalent to $\Psi_p(f_n) \rightarrow 0$ as $n \rightarrow \infty$. 
\end{proposition}

\begin{proof}
Rajput and Rosinski have shown that $\Psi_p$ is a $\varphi$-function in  \cite[Lemma 3.1]{Rajput1989spectral}.
By definition, $L^p(\dot{X}) = L^{\Psi_p}(\R^d)$ is therefore a generalized Orlicz space in the sense of Definition \ref{def:orlicz}. 
Except for the density, the  properties of the space $L^p(\dot{X})$ are then derived from Proposition \ref{prop:orlicz}.
For the density, Proposition \ref{prop:orlicz} implies that simple functions are dense in $L^p(\dot{X})$. It suffices therefore to remark that a simple function can be easily approximated by a function in $\D(\R^d)$ in the topology of $L^p(\dot{X})$ by taking a regularized version of the simple function.
\end{proof}
 
 The connection between L\'evy white noises in $\D'(\R^d)$ and independently scattered random measures established in Theorem \ref{theo:connection} allows us to apply \cite[Theorem 3.3]{Rajput1989spectral} to the L\'evy white noise and leads to the following result. 
 
 \begin{theorem}  \label{prop:continuity}
Let $\dot{X}$  be a L\'evy white noise {and $0\leq p < \infty$}. Then, the functional 
\begin{align*}
	\dot{X} :     L^p(\dot{X}) &\rightarrow L^p (\Omega) \\
			 f &\mapsto \langle \dot{X}, f  \rangle 
\end{align*}
is linear and continuous. In other words, $\dot{X}$ is a random linear functional on $L^0(\dot{X})$ and an $L^p(\Omega)$-valued random linear functional on $L^p(\dot{X})$ when $p > 0$. 
\end{theorem}}

Theorem \ref{prop:continuity} gives general conditions  on test functions for integrability with respect to a given L\'evy white noise. It therefore specifies the domain of definition of $\dot{X}$; that is, the broadest class of test functions on which $\dot X$ is a random linear functional. 
{Moreover, the space $L^p(\dot{X})$ is the largest space of test functions such that $\langle \dot{X} , f \rangle$ is well-defined and has a finite $p$th-moment.}
Once the random variable $\langle \dot{X}, f \rangle$ is well-defined, it is important to identify its characteristic function, as was done by Rajput and Rosinski in~\cite[Theorem 2.7]{Rajput1989spectral}. More generally, one uses the previous results to extend the domain of continuity and positive-definiteness of the characteristic functional of a L\'evy white noise. 

	\begin{proposition} \label{prop:continuityCF}
	For any L\'evy white noise $\dot{X}$, the characteristic functional $\CF_{\dot{X}}$ is well-defined, continuous, and positive-definite over $L^0(\dot{X})$, and is given by 
	\begin{equation} \label{eq:CFgeneral}
		\CF_{\dot{X}}(f) = \exp\left( \int_{\R^d} \psi(  f (t)) \drm t \right).
	\end{equation}
	\end{proposition}
	
	\begin{proof}
	The characteristic functional $\varphi \mapsto \CF_{\dot{X}}(\varphi) = \mathbb{E}[\mathrm{e}^{\mathrm{i} \langle \dot{X},\varphi\rangle} ]$ is  continuous over $\D(\R^d)$.
	For any $f \in L^0(\dot{X})$, we know that $\langle \dot{X},f \rangle$ is a well-defined random variable. {According to the last part of~\cite[Theorem 2.7]{Rajput1989spectral},} its characteristic function is $\xi \mapsto  \mathbb{E}[\mathrm{e}^{\mathrm{i} \xi \langle \dot{X},f\rangle}] = \exp\left( \int_{\R^d} \psi( \xi f(t)) \drm t \right)$.  We can therefore extend $\CF_{\dot{X}}$ to $L^0(\dot{X})$ by setting 
	\begin{equation*}
		\CF_{\dot{X}} (f) =  \mathbb{E}[\mathrm{e}^{\mathrm{i}   \langle \dot{X},f\rangle}] =\exp\left( \int_{\R^d} \psi(   f(t)) \drm t \right).
	\end{equation*}
	
	\textit{Positive-definiteness.} Let $N \geq 1$, $a_n \in \C$, $f_n \in L^0(\dot{X})$, $n=1,\ldots, N$. 
	The space $\D(\R^d)$ is dense in $L^0(\dot{X})$ (Proposition \ref{prop:LdotXOrlicz}), hence there exists for $N$ sequences $(\varphi_k^n)_{k \in \N}$ such that $\varphi_k^n \rightarrow f_n$ in $L^0(\dot{X})$ for $n= 1, \cdot , N$. 
	From Theorem \ref{prop:continuity}, we know that $f \mapsto \langle \dot{X}, f \rangle$ is continuous from $L^0(\dot{X})$ to $L^0 (\Omega)$. In particular, we have 
	$\CF_{\dot{X}} (\varphi^i_k -\varphi^j_k) = \mathbb{E}[\mathrm{e}^{\mathrm{i} \langle \dot{X}, \varphi_k^i -  \varphi_k^j \rangle}] {\longrightarrow} \mathbb{E}[\mathrm{e}^{\mathrm{i} \langle \dot{X}, f_i - f_j \rangle}] = \CF_{\dot{X}} (f_i - f_j) $ for every $1\leq i,j \leq N$ when $k \rightarrow \infty$.
	Finally, we have that
	\begin{align*}
		\sum_{1\leq i,j\leq N} a_i \bar{a}_j \CF_{\dot{X}} (f_i - f_j) 
		 = 
			\lim_{k\rightarrow \infty}  \sum_{1\leq i,j\leq N} a_i \bar{a}_j \CF_{\dot{X}} (\varphi^i_k -\varphi^j_k)   \geq 0
	\end{align*}
	since $\CF_{\dot{X}}$ is positive-definite over $\D(\R^d)$.

	\textit{Continuity.} Using the L\'evy--Khintchine representation \eqref{eq:LK} of $\psi$ with L\'evy triplet $(\gamma, \sigma^2, \nu)$, we have
	\begin{align} \label{eq:boundpsiPsi}
		\lvert \psi (\xi) \rvert &= \Big\lvert   \mathrm{i} \gamma \xi  + \mathrm{i} \int_{\R}  x \xi \left( \One_{\abs{x \xi }\leq 1}- \One_{\abs{x}\leq 1} \right) \nu(\drm x) + \sigma^2 \xi^2 + \int_{\R} (\mathrm{e}^{\mathrm{i} x \xi} - 1 - \mathrm{i} x \xi \One_{\abs{ x \xi } \leq 1} ) \nu (\drm x) \Big\rvert \nonumber \\
		&\leq \Big\lvert \gamma \xi + \int_{\R} x \xi \left( \One_{\abs{x \xi }\leq 1}- \One_{\abs{x}\leq 1} \right) \nu(\drm x) \Big\rvert  
				+ \sigma^2 \xi^2
				+ 2 \int_{\R} ( 1\wedge (x^2\xi^2) ) \nu(\drm x)  \leq 2 \Psi(\xi),
	\end{align}
	where we used the triangular inequality and the relation $\lvert \mathrm{e}^{\mathrm{i} y} -1 - \mathrm{i} y \One_{\abs{y} \leq 1} \rvert \leq 2(1 \wedge y^2)$ applied to $y= x \xi$. Applying \eqref{eq:boundpsiPsi} to $\xi = f(t)$ and integrating over $\R^d$, we have for every $f \in L^0(\dot{X})$,
	\begin{equation*}
		\lvert \log \CF_{\dot{X}}(f) \rvert \leq \int_{\R^d} \lvert \psi ( f(t) )\rvert \drm t \leq 2 \rVert f \lVert_{\Psi}.
	\end{equation*}
	Hence $\CF_{\dot{X}}$ is continuous at $0$. The functional $\CF_{\dot{X}}$ is positive-definite and continuous at $0$, and therefore continuous \cite[Section IV.1.2, Proposition 1.1]{ProbaBanach1987}. 
	\end{proof}

\section{Practical Determination of the Domain of Definition} \label{sec:criteria}

We provide here several criteria for the practical identification of the domain of definition of a L\'evy white noise. 
We apply our result to the Gaussian, S$\alpha$S, compound Poisson, and generalized Laplace noises.
The results presented here are new for the two latter classes of noise to the best of our knowledge.
Similar considerations are given for the domain of finite $p$th moments for $0< p \leq 2$. 
In the rest of the paper, we shall consider the spaces $L^p(\dot{X})$ for $0 \leq p \leq 2$, with the convention that $L^0(\dot{X}) = L^0(\dot{X})$ is the domain of definition of the L\'evy white noise $\dot{X}$. 

	\subsection{Basic Properties}

\begin{proposition} \label{prop:affineinvariance}
	Let $\dot{X}$ be a L\'evy white noise  and $p \geq 0$.
	\begin{itemize}
		\item \emph{Linearity:} for $f,g \in L^p(\dot{X})$ and $\lambda \in \R$, $f + \lambda g \in L^p(\dot{X})$.
		\item \emph{Invariances:} for $f \in L^p(\dot{X})$ and $H : \R^d \rightarrow \R^d$, a $C_1$-diffeomorphism, {we have $t \mapsto f(H (t)) \in L^p(\dot{X}).$}
		In particular,   the translations $f(\cdot - t_0)$, rescalings $f(b \cdot)$, and rotations $f(\mathrm{R} \cdot)$ of $f$, with $t_0 \in \R^d$, $b \neq 0$, and $\mathrm{R} \in \mathrm{SO}(d)$ an rotation matrix, are in $L^p(\dot{X})$.
	\end{itemize}
\end{proposition}

\begin{proof}
	The linearity is already known since $L^p(\dot{X})$ is a vector space by Proposition \ref{prop:LdotXOrlicz}.
	For the invariance, we simply remark that, by the substitution $u =H(t)$, $$\int_{\R^d} \Psi_p( f( H(t) )) \drm t = \int_{\R^d}   \lvert \det  J_{H^{-1}}(u) \rvert  \Psi_p (f(u))   \drm u$$ with $J_H$ the invertible Jacobian matrix of $H$. By assumption on $H$, $ \lvert \det  J_{H^{-1}}(u) \rvert $ is bounded above and below by finite strictly positive constants, implying the result.
\end{proof}

For  a Lévy white noise $\dot{X}$, the rescaling $\dot{X}(\cdot / b)$ of a factor $b\neq 0$ is the generalized random process defined by $\langle \dot{X}(\cdot / b) , \varphi \rangle = \langle \dot{X}, b^d \varphi(b \cdot) \rangle$, for $\varphi \in \D(\R^d)$. We see easily that $\dot{X}(\cdot / b)$ is itself a L\'evy white noise. Similarly, for $a \neq 0$, the generalized random process $a \dot{X}$ is still a L\'evy white noise.
We say that two generalized random processes $s_1$ and $s_2$ are \emph{independent} if their finite-dimensional marginals  are independent. 
By linearity, this is equivalent to the relation 
\begin{equation*}
\CF_{s_1 + s_2}(\varphi) = \CF_{s_1}(\varphi) \CF_{s_2}(\varphi), \qquad \text{for all } \varphi \in \D(\R^d).
\end{equation*}
If $\dot{X}_1$ and $\dot{X}_2$ are two independent L\'evy white noises, then $\dot{X}_1 + \dot{X}_2$ is also a L\'evy white noise.  

\begin{proposition} \label{prop:scalingindependence}
	Let $\dot{X}$ be a L\'evy white noise and $p \geq 0$. Then we have, for $a$ and $b$ nonzero, and $t_0 \in \R^d$,
	\begin{equation*}
		L^p(\dot{X}) = L^p(a \dot{X}) = L^p(  \dot{X} (\cdot / b))=L^p(\dot X(\cdot-t_0) )
	\end{equation*}
	If $\dot{X}_1$ and $\dot{X}_2$ are two independent L\'evy white noises, then
	\begin{equation} \label{eq:domainsumnoises}
		L^p(\dot{X}_1) \cap L^p(\dot{X}_2) \subseteq L^p(\dot{X}_1+\dot{X}_2).
	\end{equation} 
	Moreover, if at least one of the two L\'evy white noises is symmetric, then \eqref{eq:domainsumnoises} is an equality.
\end{proposition}
	
\begin{proof}
	We have $\langle \dot{X}(\cdot / b) , f \rangle = \langle \dot{X}, b^d  f (b \cdot) \rangle$, so that $f \in L^p(\dot{X}(\cdot / b)) $ if and only if $ b^d f(b \cdot) \in L^p(\dot{X})$. Then, $L^p(\dot{X})$ being a linear space that is invariant by rescaling (Proposition \ref{prop:affineinvariance}), the latter condition is equivalent to $f\in L^p(\dot{X})$, hence $L^p(  \dot{X} (\cdot / b)) = L^p(\dot{X})$. We proceed similarly for $L^p( a \dot{X})$ and $L^p(\dot X(\cdot-t_0) )$.
	
	For $i=1,2$, the L\'evy triplet of $\dot{X}_i$ ($\dot{X})$, respectively) is denoted by $(\gamma_i, \sigma^2_i, \nu_i)$ ($(\gamma,\sigma^2,\nu)$, respectively), and the corresponding exponent is $\Psi_{p,i}$ ($\Psi_{p}$, respectively).
	If $\dot{X}_1$ and $\dot{X}_2$ are independent, we have the relations $$\gamma = \gamma_1 + \gamma_2, \ \sigma^2 = \sigma_1^2 + \sigma_2^2, \ \nu = \nu_1 + \nu_2.$$
	Therefore, by the triangular inequality, we have
	\begin{align*}
		\Psi_p (\xi) & =   \Big\lvert (\gamma_1 + \gamma_2) \xi + \int_{\R} x \xi \left( \One_{\abs{x \xi }\leq 1}- \One_{\abs{x}\leq 1} \right)( \nu_1 + \nu_2) (\drm x) \Big\rvert  
				  + (\sigma^2_1 + \sigma^2_2) \xi^2  + \int_{\R} (\lvert \xi x \rvert^p  \wedge \lvert \xi x \rvert^2) ( \nu_1 + \nu_2)(\drm x)  \nonumber \\
				& \leq \Psi_{p,1}(\xi) + \Psi_{p,2}(\xi),
	\end{align*}
	which proves \eqref{eq:domainsumnoises}. When one of the noise is symmetric, for instance $\dot{X}_1$, the latter inequality is an equality since $  \gamma_1 \xi + \int_{\R}  x \xi \left( \One_{\abs{x \xi }\leq 1}- \One_{\abs{x}\leq 1} \right) \nu_1(\drm x)   =0$ and \eqref{eq:domainsumnoises} is an equality. 
\end{proof}

In general, \eqref{eq:domainsumnoises} is only an inclusion. Consider for instance the case where $\dot{X}_1$ and $\dot{X}_2$ have L\'evy triplet $(1,1,0)$ and $(-1,0,0)$ respectively, meaning that $\dot{X}_1$ is a Gaussian white noise with drift $\gamma =1$ and $\dot{X}_2$ a pure drift $\gamma = -1$. Then, $\dot{X}_1$ and $\dot{X}_2$ are clearly independent, and $\dot{X}_1 + \dot{X}_2$ is a Gaussian white noise with no drift. Therefore, $L^p(\dot{X}_1 + \dot{X}_2) = L^2(\R^d)$ but $L^p(\dot{X}_1) \cap L^p(\dot{X}_2) = L^2(\R^d) \cap L^1(\R^d)$ (see Section \ref{sec:GaussAndDrift} for more details on the determination of those domains). 

For $\gamma \in \R$ and $\nu$ a L\'evy measure, we set
 \begin{equation*}
 m_{\gamma,\nu}(\xi) = \left\vert \gamma \xi + \int_{\R}  x \xi \left( \One_{\abs{x \xi }\leq 1}- \One_{\abs{x}\leq 1} \right)  \nu(\drm x) \right\rvert.
 \end{equation*}
The next  result is taken from \cite{Rajput1989spectral}. 

\begin{proposition}[Reduction to the symmetric case without Gaussian part] \label{prop:reduction}
	Let $(\gamma, \sigma^2, \nu)$ be a L\'evy triplet. We also denote by $\nu_{\mathrm{sym}}$ the symmetrization of $\nu$. {Thereafter, $\dot{X}$, $\dot{X}_2$, and $\dot{X}_{\mathrm{sym}}$ are L\'evy white noises with respective L\'evy triplets $(\gamma, \sigma^2, \nu)$,  $(\gamma, 0, \nu)$, and   $(0,\sigma^2,\nu_{\mathrm{sym}})$.}
Then, we have the following relations for $p \geq 0$:
	\begin{itemize}
		\item If $\sigma^2 \neq 0$, then 
		\begin{align} \label{eq:LpwithoutGaussian}
			L^p(\dot{X}) &= L^2(\R^d) \cap L^p(\dot{X}_2).
		\end{align}
		\item In any case,
		\begin{align} \label{eq:Lpsymmetrized}
			L^p(\dot{X}) &= L^p(\dot{X}_{\mathrm{sym}}) \cap \{f \in L^0(\dot{X}), \ \int_{\R^d} m_{\gamma,\nu}(f(t)) \drm t < \infty \}.
		\end{align}			
	\end{itemize}
\end{proposition}
\begin{proof}
We can decompose $\dot{X} = \dot{X}_2 + \dot{X}_{\mathrm{Gauss}}$, where $\dot{X}_2$ and $ \dot{X}_{\mathrm{Gauss}}$ are independent with respective L\'evy triplets $(\gamma, 0 ,\nu)$ and $(0,\sigma^2,0)$. Then, $ \dot{X}_{\mathrm{Gauss}}$ is a Gaussian white noise, for which $L^p( \dot{X}_{\mathrm{Gauss}}) = L^2(\R^d)$. We apply \eqref{eq:domainsumnoises} with equality ($ \dot{X}_{\mathrm{Gauss}}$ being symmetric) to obtain  \eqref{eq:LpwithoutGaussian}.
Finally, \eqref{eq:Lpsymmetrized} is a reformulation of \cite[Proposition 2.9]{Rajput1989spectral}.	
\end{proof}

Based on Proposition \ref{prop:reduction}, we restrict our attention to symmetric L\'evy white noises without Gaussian part.
We first reduce to the case $ \sigma^2=0$ thanks to   \eqref{eq:LpwithoutGaussian}.  
The only remaining part to deduce the general case from the symmetric one is the identification of functions $f$ satisfying   $ \int_{\R^d} m_{\gamma,\nu}(f(t)) \drm t < \infty$. 
Primarily, for non-symmetric noise, this usually relies on $L^1$-type conditions, but we leave this topic open for further investigations.

	\subsection{The spaces $L^{p_0,p_\infty}(\R^d)$}
        
        We introduce the family of function spaces that generalize the $L^p$-spaces for $0<p<\infty$.
They will be identified later on as the domain of definition of important classes of L\'evy white noises.
	We first give some notations. For $0\leq p_0 , p_\infty < \infty$, we set
        \begin{align*}
        		\rho_{p_0 , p_\infty}(\xi) &:=  \lvert \xi \vert^{p_0} \One_{\lvert \xi \rvert > 1} +  \lvert \xi \vert^{p_\infty} \One_{\lvert \xi \rvert \leq  1}, \\
        		\rho_{\log,p_\infty}(\xi) &:= (1+ \log  \lvert \xi \rvert) \One_{\lvert \xi \rvert > 1} +  \lvert \xi \vert^{p_\infty} \One_{\lvert \xi \rvert \leq  1} .
	\end{align*}
	with the convention that {$0^0=0$}.
        
        \begin{definition} \label{def:Lpq}
        		For $0\leq p_0, p_\infty < \infty$, we define
		\begin{align*}
			L^{p_0, p_\infty}(\R^d) &= \left\{ f   \text{ measurable}, \  \rho_{p_0,p_\infty} (f)  := \int_{\R^d} \rho_{p_0,p_\infty} (f(t)) \drm t  < \infty  \right\}, \\
			L^{\log, p_\infty}(\R^d) &= \left\{ f  \text{ measurable},  \ \rho_{\log,p_\infty} (f)   :=  \int_{\R^d}  \rho_{\log,p_\infty} (f(t)) \drm t < \infty  \right\}.
		\end{align*}
        \end{definition} 
    For $p>0$, we have $L^{p,p} (\R^d) = L^p (\R^d)$. 
    Roughly speaking, $p_0$ measures the local integrability of a function, while $p_\infty$ indicates the asymptotic one.
This is illustrated by the following example. For $\alpha,\beta > 0$, the function $f(t) = \abs{t}^{-\alpha} \One_{\abs{t} < 1} +  \abs{t}^{-\beta} \One_{\abs{t} \geq 1}$ is such that
    \begin{align*}
    	\rho_{p_0,p_\infty}(f) &= \int_{\R^d} \left( \lvert f(t) \vert^{p_0} \One_{\lvert f(t) \rvert > 1} +  \lvert f(t) \vert^{p_\infty} \One_{\lvert f(t) \rvert \leq  1} \right)  \drm t  = \int_{\abs{t} < 1} \abs{t}^{- p_0 \alpha} \drm t +  \int_{\abs{t}  \geq  1} \abs{t}^{- p_\infty\beta} \drm t.
    \end{align*}
    Therefore, $f$ is in $L^{p_0,p_\infty}(\R^d)$ if and only if
   \begin{equation*}
  \alpha < \frac{d}{p_0} \text{ and } \beta > \frac{d}{p_\infty}.
   \end{equation*}
   The first inequality effectively refers to the integrability of $f$ at the origin (or local integrability), while the second covers its asymptotic integrability.

       As we did in Section~\ref{subsec:extension}  with the spaces $L^0(\dot{X})$ and  $L^p(\dot{X})$, we rely on generalized Orlicz spaces~\cite[Chapter X]{Rao1991theory} to identify the structure of the spaces $L^{p_0,p_\infty}(\R^d)$. 

	\begin{proposition}  \label{prop:rhopqisnice}
		We fix  $p_0 \geq 0$ and $p_\infty > 0$.
 		The functions $\rho_{p_0,p_\infty}$ and $\rho_{\log,p_\infty}$ are  $\Delta_2$-regular $\varphi$-functions.
	\end{proposition}

	\begin{proof}
		To simplify the notation, we write $\rho = \rho_{p_0,p_\infty}$ in this proof.
		The function $\rho$ is continuous , non-decreasing, symmetric, and vanishes at the origin (since $p_\infty \neq 0$).
		It is therefore a $\varphi$-function.
		
		Then, we have the following decomposition
		\begin{equation*} \label{eq:boundDelta2_second}
			\rho( 2 \xi ) 
			= 2^{p_0} \abs{\xi}^{p_0} \One_{\abs{\xi} > 1} 
			+ 2^{p_0} \abs{\xi}^{p_0} \One_{1/2 < \abs{\xi} \leq 1} 
			+ 2^{p_\infty} \abs{\xi}^{p_\infty} \One_{\abs{\xi} \leq 1/2}. 
		\end{equation*}		
		For $1/2 \leq \abs{\xi} \leq 1$, we have that $\abs{\xi}^{p_0 - p_\infty} \leq \max (2^{p_\infty - p_0}, 1) $. Therefore, we have that
		\begin{align*}
			\rho( 2 \xi )
			& \leq 2^{p_0} \abs{\xi}^{p_0} \One_{\abs{\xi} > 1} 
			+ 2^{p_0}  \max (2^{p_\infty, p_0}, 1)  \abs{\xi}^{p_\infty} \One_{1/2 < \abs{\xi} \leq 1} 
			+ 2^{p_\infty} \abs{\xi}^{p_\infty} \One_{\abs{\xi} \leq 1/2}   \leq \max(2^{p_0}, 2^{p_\infty}) \rho(\xi).
		\end{align*} 
		Therefore, $\rho$ is $\Delta_2$-regular.
		The proof for $\rho_{\log,p_\infty}$   is very similar.
	\end{proof}
	
	Proposition \ref{prop:rhopqisnice} coupled with Proposition \ref{prop:orlicz} allow us to identify the structure of the spaces $L^{p_0,p_\infty}(\R^d)$ and $L^{\log,p_\infty}(\R^d)$
    
    \begin{proposition}  \label{prop:structureLpq}
    	We fix  $ p_0 \geq 0$ and $p_\infty  > 0$.
	Then, $L^{p_0,p_\infty}(\R^d) = L^{\rho_{p_0,p_\infty}}(\R^d)$ is a generalized Orlicz space associated to the $\varphi$-function $\rho_{p_0,p_\infty}$. It is in particular a complete linear metric space for the F-norm
	\begin{equation*}
		\lVert f \rVert_{\rho_{p_0,p_\infty}} := \inf \{\lambda > 0 , \ \rho_{p_0,p_\infty}(f / \lambda) \leq \lambda \}.
	\end{equation*}
	Finally, simple functions are dense in $L^{p_0,p_\infty}(\R^d)$.
	The same conclusions occur for $L^{\log,p_\infty}(\R^d)$.
	\end{proposition}
	
	The following embeddings are easily deduced by bounding the F-norm of the considered function spaces.

\begin{proposition}	 \label{prop:embeddingsLpq}
	We fix  $ p_0 \geq 0$ and $p_\infty  > 0$.
\begin{enumerate}
		{	\item If  $0\leq p_1 \leq p_2<\infty$, we have the   embedding 
			$L^{p_2,p_\infty}(\R^d) \subseteq L^{p_1,p_\infty}(\R^d)$.
		\item If $0< p_1 \leq p_2<\infty$, we have the  embedding 
		$L^{p_0,p_1}(\R^d) \subseteq L^{p_0,p_2}(\R^d)$.
		\item Conditions 1. and  2. remain true by changing $p_0$ to $\log$ and we have the  embeddings, for any $0< p_0 \leq 2$ and $0\leq p_\infty \leq 2$, $L^{p_0,p_\infty}(\R^d) \subseteq L^{\log,p_\infty}(\R^d) \subseteq L^{0,p_\infty}(\R^d)$.}
	\end{enumerate}
    \end{proposition}	
 
In Propositions \ref{prop:structureLpq} and \ref{prop:embeddingsLpq},
we restricted ourselves to the case when $p_\infty \neq 0$. 
The reason is that $\rho_{p_0,0}(0) \neq 0$, 
so that $\rho_{p_0,0}$ is not a $\varphi$-function. 
Therefore, we do not define a generalized Orlicz space in the sense of Rao and Ren \cite{Rao1991theory}.
The space $L^{p_0,0}(\R^d)$ can be described as follows:
It is the space of functions in $L^{p_0} (\R^d)$ whose support has a finite Lebesgue measure. 
We do not specify any topological structure on those spaces, since they will not appear as the domain of definition of any Lévy white noise.
However, the space $L^{2,0}(\R^d)$ will play a role as a common subspace to all the domains of definition of the L\'evy white noises (see Proposition \ref{prop:limitspaces}).
    
	\subsection{Criteria for the Determination of the Domain of Definition}\label{subsec:criteria}

In this section, we consider a symmetric white noise $\dot{X}$ without Gaussian part and with symmetric L\'evy measure $\nu$. 
In particular, for $p \geq 0$, the $p$th-order Rajput--Rosinski exponent in \eqref{eq:Psi}  simply becomes
\begin{align}
  	 \Psi_p(\xi)    =  \lvert \xi \rvert^2 \int_{|x| \leq 1/|\xi|} \lvert x \rvert^2 \nu (\drm x) +  \lvert \xi \rvert^p  \int_{|x|> 1/ \lvert \xi \rvert} \lvert x \rvert^p \nu (\drm x). \label{eq:psimmetric}
\end{align}
The first criterion is applicable as soon as we are able to estimate the behavior of the function $\Psi_p$ at the origin and/or at infinity.

\begin{proposition}[Criteria for the determination of the domain of definition] \label{prop:criteria}
	Let $\dot{X}$ be a symmetric L\'evy white noise without Gaussian part and $0 \leq p \leq 2$.
	\begin{enumerate}
	\item Assume that $\Psi_p(\xi) \leq C \rho_{p_0,p_\infty}(\xi)$ for some constant $C>0$ and every $\xi$, then we have the   embedding
	\begin{equation} \label{eq:practicalX1}
		L^{p_0,p_\infty}(\R^d) \subseteq L^p(\dot{X}). 
	\end{equation}
	\item Assume that $\rho_{p_0,p_\infty}(\xi) \leq C  \Psi_p(\xi)  $ for some constant $C>0$ and every $\xi$, then we have the   embedding
	\begin{equation} \label{eq:practicalX2}
		L^p(\dot{X}) \subseteq L^{p_0,p_\infty}(\R^d)  . 
	\end{equation}
	\item Assume that $\Psi_p(\xi) \underset{0}{\sim} A \lvert \xi \rvert^{p_\infty}$ and $\Psi_p(\xi) \underset{\infty}{\sim} B \lvert \xi \rvert^{p_0}$,
	then \begin{equation}  \label{eq:practicalX3}  L^p(\dot{X}) = L^{p_0,p_\infty}(\R^d) . \end{equation}
	\item  The same holds with $L^{\log,p_\infty}(\R^d)$ instead of $L^{p_0,p_\infty}(\R^d)$  if we replace $\lvert \xi \rvert^{p_0}$ by $ \log \lvert \xi\rvert $.  

	\end{enumerate}
	
\end{proposition}        

\begin{proof}
	The condition $\Psi_p (\xi) \leq C \rho_{p_0,p_\infty} (\xi)$ implies that, for any function $f \in L^{p_0,p_\infty}(\R^d) $, we have
	\begin{equation*}
		\lVert f \rVert_{\Psi_p} = \int_{\R^d} \Psi_p(f(t)) \drm t \leq C \int_{\R^d} \rho_{p_0,p_\infty} (f(t) )\drm t = C  \lVert f \rVert_{p_0,p_\infty}.
	\end{equation*}
	Therefore, the identity map is continuous from    $L^{p_0,p_\infty}(\R^d)$ to $L^p(\dot{X})$ proving \eqref{eq:practicalX1}. The proof of \eqref{eq:practicalX2} is very similar.
	For the last point, we remark that the two functions $\Psi_p$ and $\rho_{p_0,p_\infty}$ do not vanish for $\xi \neq 0$,  are continuous, and are equivalent at $0$ and infinity. Hence, there exists two constants  such that
	\begin{equation*}
		C_1  \rho_{p_0,p_\infty}(\xi) \leq  \Psi_p(\xi) \leq C_2 \rho_{p_0,p_\infty}(\xi).
	\end{equation*}
	We then apply \eqref{eq:practicalX1}  and \eqref{eq:practicalX2}  to obtain \eqref{eq:practicalX3} 
\end{proof}

Note that the local integrability of test functions (parameter $p_0$) is linked with the asymptotic behavior of $\Psi_p$, while the asymptotic integrability (parameter $p_\infty$) is linked to the behavior of $\Psi_p$ at $0$. 

If we know that the L\'evy measure has some finite moments, then we obtain new information on the domain of definition of the L\'evy white noise.   For  $p,q \geq 0$, we set
\begin{equation} \label{eq:genmoments}
	m_{p,q} (\nu) := \int_{\R} \rho_{p,q} (x) \nu(\drm x) = \int_{\abs{x}> 1} \abs{x}^p \nu (\drm x) + \int_{\abs{x} \leq 1} \abs{x}^q \nu (\drm x),
\end{equation}
called  the \emph{generalized moments} of $\nu$. 
Then, $\nu$ being a L\'evy measure, we have that $m_{0,2} (\nu) < \infty$. 

Consider the L\'evy process $X$, together with its corresponding L\'evy white noise $\dot{X}$, with L\'evy triplet $(0,0,\nu)$.
Then, $\dot{X}$  has finite $p$th moments if and only if $m_{p,2} (\nu) < \infty$ \cite[Theorem 25.3]{Sato1994levy}.
The quantity
\begin{equation*}
	\beta_0 := \sup\{ 0\leq p \leq 2, \ m_{p,2}(\nu) < \infty \}
\end{equation*}
is called the \emph{Pruitt index} and was introduced in \cite{Pruitt1981growth} to study the asymptotic behavior of L\'evy processes.
It measures the growth rate of $X$ at infinity \cite[Section 5.3]{Bottcher2014levy} and is therefore strongly related with the required rate of decay for the control of the Besov regularity of the L\'evy white noise $\dot{X}$ \cite[Theorem 3]{Fageot2017multidimensional}.

In contrast, the \emph{Blumenthal-Getoor index}, defined as
\begin{equation*}
	\beta_\infty := \inf\{ 0\leq q \leq 2, \ m_{0,q}(\nu) < \infty \},
\end{equation*}
relies on the local regularity of $X$ and $\dot{X}$. 
This can be formulated in terms of the strong variation of $X$ \cite[Section 5.4]{Bottcher2014levy} or the local Besov regularity of $X$  (see \cite{Schilling1997Feller}) and of  $\dot{X}$ (see \cite[Corollary 3]{Fageot2017multidimensional}).

In accordance with the previous remarks, the generalized moments of a L\'evy measure $\nu$ have important interpretations for the local and asymptotic behaviors of the L\'evy white noise and the corresponding L\'evy process.

\begin{proposition} \label{lemma:controlPsip}
	Let $\dot{X}$ be a symmetric L\'evy white noise with L\'evy measure $\nu$ and without Gaussian part.
	\begin{itemize}
	\item We assume that $m_{p,2} (\nu) < \infty$ for some $0 \leq p \leq 2$.
	Then, we have, for any $\xi \in \R$, that
	\begin{equation} \label{eq:PsipAlphap}
		m_{p,2}(\nu)  \rho_{p,2}(\xi)  \leq \Psi_p(\xi) \leq  m_{p,2}(\nu)   \rho_{2,p} (\xi).
	\end{equation}
	\item We assume that $m_{p,2} (\nu) < \infty$ for some $p \geq 2$.
	Then, we have, for any $\xi \in \R$, that
	\begin{equation} \label{eq:PsipAlphapbiggerthan2}
		m_{p,2}(\nu)  \rho_{2,p}(\xi)  \leq \Psi_p(\xi) \leq  m_{p,2}(\nu)   \rho_{p,2} (\xi).
	\end{equation}
	\item If   $m_{p_\infty, p_0} (\nu) < \infty$ for some $0\leq p_0 \leq 2$, $0 < p_\infty < \infty$ and if $p\leq p_0, p_\infty$, then
	\begin{equation} \label{eq:PsipAlphap0pinf}
		 \Psi_p(\xi) \leq  m_{\min(p_\infty,2), p_0}(\nu)   \rho_{p_0,\min(p_\infty,2)} (\xi).
	\end{equation}
	\end{itemize}
\end{proposition}

\begin{proof}
	All the inequalities will be obtained by exploiting the position of $\abs{x}$, $\abs{\xi}$, or $\abs{x \xi}$ with respect to $1$.
	{The proofs for the upper and lower bounds in \eqref{eq:PsipAlphap} and \eqref{eq:PsipAlphapbiggerthan2} and the upper bound in \eqref{eq:PsipAlphap0pinf} are very similar, we therefore focus on the last one.
We first assume that $\abs{\xi} \leq 1$. Then, using \eqref{eq:psimmetric}, we  decompose $\Psi_p$ as 
	\begin{equation} \label{eq:developPsi}
		\Psi_p(\xi) = 
		\int_{\lvert x \rvert \leq 1} \lvert x \xi  \rvert^2 \nu(\drm x) 
		+ 
		\int_{1< \lvert x \rvert \leq \frac{1}{\lvert \xi \rvert }} \lvert x \xi  \rvert^2 \nu(\drm x) 
		+ 
		 \int_{\lvert x \rvert > \frac{1}{\lvert \xi \rvert }} \lvert x   \xi \rvert^p\nu(\drm x) .		
	\end{equation}
Applying this relation to $p \leq p_\infty \leq 2$, we deduce that
	\begin{align*}
	\Psi_p(\xi) 
		& \leq
		\int_{\lvert x \rvert \leq 1} \lvert x \rvert^2   \lvert \xi \rvert^{p_\infty} \nu(\drm x) 
		+ 
		\int_{1< \lvert x \rvert \leq \frac{1}{\lvert \xi \rvert }} \lvert  x \xi  \rvert^{p_\infty} \nu(\drm x) 
		+ 
		\int_{\lvert x \rvert > \frac{1}{\lvert \xi \rvert }} \lvert x \xi \rvert^{p_\infty} \nu(\drm x)    = m_{p_\infty,2}(\nu) \abs{\xi}^p. 
	\end{align*}
	If now $p_\infty > 2$, we have, still for $\abs{\xi} \leq 1$, that
	\begin{align*}
	\Psi_p(\xi) 
		& \leq
		\int_{\lvert x \rvert \leq 1} \lvert x \rvert^2   \lvert \xi \rvert^{2} \nu(\drm x) 
		+ 
		\int_{1< \lvert x \rvert \leq \frac{1}{\lvert \xi \rvert }} \abs{x}^{p_\infty} \abs{\xi}^{2} \nu(\drm x) 
		+ 
		\int_{\lvert x \rvert > \frac{1}{\lvert \xi \rvert }} \abs{   \xi }^{2} \nu(\drm x)  	 =  m_{2,2}(\nu) \abs{\xi}^2. 
	\end{align*}
	We deduce that  $\Psi_p(\xi) \leq  m_{\min(p_\infty,2),2}(\nu)  \abs{\xi}^{\min(p_\infty , 2)}$. \\
	
	Assume now that $\abs{\xi} > 1$. Then, we use the decomposition 
 	\begin{equation} \label{eq:developPsi2}
 		\Psi_p(\xi) = 
 		\int_{\lvert x \rvert \leq \frac{1}{\abs{\xi}}} \lvert x \xi  \rvert^2 \nu(\drm x) 
 		+ 
 		\int_{\frac{1}{\abs{\xi}}< \abs{x} \leq 1} \lvert x \xi  \rvert^p \nu(\drm x) 
 		+ 
 		 \int_{\lvert x \rvert > 1} \lvert x   \xi \rvert^p\nu(\drm x) .	
 	\end{equation}
	When, $p \leq p_0 \leq 2$ and $p < p_\infty$, we therefore have that 
	\begin{align*}
	\Psi_p(\xi) 
		& \leq
		\int_{\lvert x \rvert \leq \frac{1}{\abs{\xi}}} \abs{ x \xi}^{p_0} \nu(\drm x) 
		+ 
		\int_{\frac{1}{\abs{\xi}}< \abs{x} \leq 1} \abs{ x\xi }^{p_0} \nu(\drm x) 
		+ 
		 \int_{\lvert x \rvert > 1} \abs{ x }^{\min(p_\infty,2)} \abs{ \xi }^{p_0} \nu(\drm x)	  = m_{\min(p_\infty,2), p_0}(\nu) \abs{\xi}^{p_0}.
	\end{align*}	
	Remarking that $m_{\min(p_\infty,2) , 2}(\nu) \leq m_{\min(p_\infty,2) , p_0}(\nu)$ and combining the bounds for $\abs{\xi} \leq 1$ and $\abs{\xi} > 1$,  we deduce \eqref{eq:PsipAlphap0pinf}. }
\end{proof}

\begin{proposition} \label{prop:limitspaces}
Let $\dot{X}$ be a L\'evy white noise with L\'evy measure $\nu$ and $p >0$. Then, the following statements hold.
 \begin{itemize}
     \item In the general case, we have
	\begin{align}\label{eq:limitL0}
		L^{2,0}(\R^d) \subseteq L^0(\dot{X}) \subseteq L^{0,2}(\R^d),
	\end{align}
	\item If $0< p \leq 2$ and $\dot{X}$ is symmetric such that $m_{p,2}(\nu) < \infty$, then
\begin{equation} \label{eq:limitLp}
	L^{2,p}(\R^d) \subseteq L^{p}(\dot{X}) \subseteq L^{p,2}(\R^d).
\end{equation}
    
    \item If $p \geq 2$ and $\dot{X}$ is symmetric such that $m_{p,2}(\nu) < \infty$, then
    \begin{equation} \label{eq:limitLpbigp}
	    L^{p,2}(\R^d) \subseteq L^p(\dot{X}) \subseteq L^{2,p}(\R^d).
    \end{equation}
    {In particular,   for any symmetric   finite-variance L\'evy white noise 
$L^{2}(\dot{X}) = L^2(\R^d)$.}
    \item If $\dot{X}$ is symmetric without Gaussian part and such that $m_{p_\infty, p_0}(\nu) < \infty$ with $0\leq p \leq p_0, p_\infty \leq 2$,  then
	\begin{equation} \label{eq:limitmoreprecise}
	 L^{p_0,p_\infty}(\R^d) \subseteq L^p(\dot{X}) .
	\end{equation}
 \end{itemize}
 \end{proposition}

\begin{proof}
	When $\dot{X}$ is symmetric without Gaussian part, \eqref{eq:limitL0} and \eqref{eq:limitLp} are directly deduced from \eqref{eq:PsipAlphap} by taking $p=0$ and $p$ general, respectively. 
	Adding a Gaussian part  does not change the conclusions since $L^{2,p} (\R^d) \subseteq L^{p}(\dot{X}_{\mathrm{Gauss}}) = L^2(\R^d) \subseteq L^{p,2}(\R^d)$ for all $0 \leq p \leq 2$ and thanks to \eqref{eq:LpwithoutGaussian}.
	
	We now consider a general Lévy white noise  $\dot{X}$ with L\'evy triplet  $(\gamma,\sigma^2,\nu)$  and $w_{\mathrm{sym}}$ its symmetric version with triplet  $(0,\sigma^2,\nu_{\mathrm{sym}})$. 
We already know that 
$L^{2,0}(\R^d) \subseteq  L^0(\dot{X}_{\mathrm{sym}}) \subseteq L^{0,2}(\R^d)$. 
Moreover, from \eqref{eq:Lpsymmetrized}, we know that 
\begin{equation} \label{eq:rappel}
	L^0(\dot{X}) = L^0(\dot{X}_{\mathrm{sym}}) \cap \left\{f \in L^0(\dot{X}), \ \int_{\R^d} m_{\gamma,\nu}(f(t)) \drm t < \infty \right\}.
\end{equation}
First, we have that $L^0(\dot{X}) \subseteq L^0(\dot{X}_{\mathrm{sym}}) \subseteq L^{0,2}(\R^d)$. Second, due to \eqref{eq:rappel},  it is sufficient to prove that 
$$L^{2,0} (\R^d) \subseteq  \{f \in L^0(\dot{X}), \ \int_{\R^d} m_{\gamma,\nu}(f(t)) \drm t < \infty \}$$
to deduce that $L^{2,0}(\R^d) \subseteq L^0(\dot{X})$. We remark that, for $\lvert \xi \rvert \leq 1$,
\begin{align*}
  m_{\gamma,\nu}(\xi) &= \left\lvert \gamma \xi + \int_{1\leq \lvert x \rvert \leq \frac{1}{\lvert \xi\rvert}}  \xi x  \nu({\drm x}) \right\rvert 
   \leq  \abs{\gamma \xi} + \int_{1\leq \lvert x \rvert \leq \frac{1}{\lvert \xi\rvert}}  \nu({\drm x})   \leq \rvert  \gamma \lvert + \int_{1 \leq \lvert x \rvert} \nu(\drm x),
\end{align*}
and that, for $\lvert \xi \rvert > 1$, 
\begin{align*}
 m_{\gamma,\nu}(\xi) &=  \left\lvert \gamma \xi + \int_{\frac{1}{\lvert \xi\rvert} \leq \lvert x \rvert \leq 1}   \xi  x  \nu(\drm x) \right\rvert 
 \leq \abs{ \gamma \xi } + \int_{\frac{1}{\lvert \xi\rvert} \leq \lvert x \rvert \leq 1} \lvert \xi x \rvert^2 \nu(\drm x)    \leq \left( \rvert \gamma \lvert + \int_{ \lvert x\rvert \leq 1} x^2  \nu(\drm x)\right) \xi^2.
\end{align*}
Therefore, we have $  m_{\gamma,\nu}(\xi)   \leq C \rho_{2,0}(\xi)$ for some constant $C$, which implies that $L^{2,0}(\R^d)$ is included into  $\{f \in L^0(\dot{X}), \ \int_{\R^d} m_{\gamma,\nu}(f(t)) \drm t < \infty \}$, as expected.
Finally, \eqref{eq:limitmoreprecise}   is a direct consequence of \eqref{eq:PsipAlphap0pinf}.  
\end{proof}

\textit{Remarks.}
\begin{itemize}
\item[(i)] The embeddings \eqref{eq:limitL0} inform on  the extreme cases.
In particular, a function in $L^{2,0}(\R^d)$---the space of  functions  in $L^2(\R^d)$ whose support has a finite Lebesgue measure---can be applied to any   Lévy white noise.
This includes the indicator functions $\One_{B}$ with $B$ a Borel set with finite Lebesgue measure or the Daubechies wavelets that are compactly supported and in $L^2(\R^d)$. 
We shall see that finite-variance compound Poisson noises reach the largest possible domain of definition $L^{0,2}(\R^d)$ (see Section \ref{subsubsec:poisson}). 

{The embeddings \eqref{eq:limitLp} and \eqref{eq:limitLpbigp} complement the general statement \eqref{eq:limitL0}. It implies in particular that $\langle \dot{X} , f \rangle$ has a finite $p$th-moment as soon as $\dot{X}$ has.}

{\item[(ii)] The relation \eqref{eq:limitmoreprecise} provides general sufficient conditions ensuring the finiteness of moments of $\langle \dot{X}, f \rangle$. It will play an important role when identifying compatibility condition between a whitening operator and a L\'evy white noise in Section \ref{sec:application}.}

\end{itemize}

	\subsection{Examples} \label{subsec:examples}
        
        In this section, we consider subfamilies of infinitely divisible laws that define important classes of L\'evy white noises.
        For these different classes, we specify the domain of definition $L^0(\dot{X})$ and the domains $L^p(\dot{X})$ of the considered L\'evy white noises.
       
        The function $\One_{[0,1]^d} \in L^{2,0}(\R^d) $ is in the domain of definition of every L\'evy white noise. Moreover, according to  \eqref{eq:CFgeneral}, a L\'evy white noise with characteristic exponent $\psi$ is such that
        \begin{equation*}
        		\Phi_{\langle \dot{X}, \One_{[0,1]^d} \rangle }(\xi) = \exp\left( \int_{\R^d} \psi ( \xi \One_{[0,1]^d} (t) ) \drm t\right) = \exp (\psi(\xi))
        \end{equation*}
        since $\psi(0) = 0$. Therefore, the characteristic exponent of the L\'evy white noise is also the characteristic exponent of the random variable $\langle \dot{X}, \One_{[0,1]^d} \rangle$. We take the convention that the terminology for the  law of this random variable is inherited by the   L\'evy white noise. For instance, a  white noise is said to be Gaussian if the random variable $\langle \dot{X}, \One_{[0,1]^d} \rangle$ is Gaussian.
        
        We   illustrate how to deduce the domain of definitions of Gaussian, S$\alpha$S, generalized Laplace, and compound Poisson white noises.

        		\subsubsection{Gaussian White Noises and Pure Drift White Noises} \label{sec:GaussAndDrift}

		The Gaussian white noise of variance $\sigma^2$ is characterized by the L\'evy triplet $(0,\sigma^2,0)$. With Proposition \ref{integrable}, we directly obtain that, for every $0\leq p \leq 2$, 
	$L^p(\dot{X}_{\mathrm{Gauss}}) = L^2(\R^d)$. 
		Based on these   considerations and Proposition \ref{prop:reduction}, we shall consider L\'evy triplets with $\sigma^2 = 0$ from now. 

		Similarly, the pure drift white noise $\dot{X}_{\mathrm{drift}}$ with mean $\gamma$ is defined from its triplet $(\gamma, 0,0)$. We have in that case that $\dot{X}_{\mathrm{drift}} = \gamma$ almost surely and $\dot{X}_{\mathrm{drift}}$ is a constant---therefore non stochastic---process. Then, we have 
	$L^p(\dot{X}_{\mathrm{drift}}) = L^1(\R^d)$
		 for every $0\leq p \leq 2$.
		
		If now $\dot{X} = \dot{X}_{\mathrm{Gauss}} + \dot{X}_{\mathrm{drift}}$ has L\'evy triplet $(\gamma, \sigma^2,0)$ with $\gamma$ and $\sigma^2 \neq 0$, then the domain is, due to \eqref{eq:domainsumnoises} with equality,
	$L^p(\dot{X}) =   L^2(\R^d)\cap  L^1(\R^d)$.
		
		\subsubsection{Non-Gaussian S$\alpha$S White Noises}
		
		Stable random variables are an important subclass of infinitely divisible random variables. Extensive details on S$\alpha$S random variables and random processes can be found in \cite{Taqqu1994stable}.

	The extension of the symmetric $\alpha$-stable noise to an independently scattered random measure is in fact an $\alpha$-stable random measure in the sense of \cite{Taqqu1994stable}. The identification of the space of deterministic integrable functions has already been carried out in this context in \cite{Taqqu1994stable}, and we merely re-state the result and prove it within our framework.
 
		We fix $0< \alpha < 2$.
		A random variable is symmetric-$\alpha$-stable (S$\alpha$S) if its characteristic function can be written as  $\mathrm{e}^{- \gamma \lvert \xi \rvert^\alpha}$ for some $\gamma >0$. For simplicity, we should only consider $\gamma = 1$ thereafter, since a different $\gamma$ will not change the domain of definition according to Proposition \ref{prop:scalingindependence}. A S$\alpha$S  white noise $\dot{X}_\alpha$ is a L\'evy white noise such that $\langle \dot{X}_\alpha , \One_{[0,1]^d} \rangle$ is a S$\alpha$S random variable. Its characteristic functional is given for $\varphi \in \D(\R^d)$ by $\CF_{\dot{X}_\alpha}(\varphi) = \exp \left( - \lVert \varphi \rVert^{\alpha}_\alpha \right)$ (see for instance \cite[Section 4.2.0]{Unser2014sparse}).
		
		\begin{proposition}\label{prop:salphas} 
			Let $0 < \alpha < 2$. Then, for every $0\leq p < \alpha$, we have
			\begin{equation*}
				L^p(\dot{X}_{\alpha}) = L^{\alpha}(\R^d).
			\end{equation*}
			For $p \geq \alpha$, we have $L^p(\dot{X}_\alpha) = \{0\}$.
		\end{proposition}

\begin{proof}
	The L\'evy measure of $\dot{X}_\alpha$ is  $\nu(\drm x) = \frac{C_\alpha}{\lvert x \rvert^{\alpha + 1}} \drm x$ with $C_\alpha$ a constant. A non-trivial S$\alpha$S random variable has an infinite $p$th-moment for $p\geq \alpha$, and for every $f \in L^0(\dot{X}_\alpha)$, $\langle \dot{X}, f \rangle$ is a S$\alpha$S random variable. Hence $L^p (\dot{X}) = \{0\}$ for $p\geq \alpha$. The case of interest is  therefore  $0\leq p < \alpha$. Then, from \eqref{eq:psimmetric},
	\begin{align*}
		\Psi_{p}(\xi) &= 2C_\alpha \int_{0}^{1/\lvert \xi \rvert} \frac{\xi^2}{x^{\alpha + 1}} \drm x + 2C_\alpha \int_{1 / \lvert \xi \rvert} \frac{\lvert \xi \rvert^p}{x^{\alpha + 1 - p}} \drm x  = 2C_\alpha \lvert \xi \rvert^\alpha \left( \int_0^1 \frac{\drm y}{y^{\alpha -1}} + \int_{1}^\infty \frac{\drm y }{y^{\alpha + 1 - p}} \right) \\
				&= \left( \frac{2(2-p)C_\alpha}{(2-\alpha)(\alpha-p)} \right) \lvert \xi \rvert^\alpha.
	\end{align*}
 Finally, the result follows from Proposition \ref{prop:criteria}. 
\end{proof}

		\subsubsection{Generalized Laplace White Noises}
		Our goal is to study the   Laplace white noise, for which $\langle \dot{X} , \One_{[0,1]^d} \rangle$ follows a Laplace law. It requires to introduce the family of generalized Laplace laws. We follow here the terminology of \cite[Section 4.1.1]{Koltz2001laplace} and consider only the symmetric case.
		
		A random variable $Y$ is called a \emph{generalized Laplace random variable} if its characteristic function can be written as
		\begin{equation*}
			\Phi(\xi) = \frac{1}{(1 + \frac{1}{2} \sigma^2 \xi^2)^\tau} = \exp \left( - \tau \log (1 + \frac{1}{2} \sigma^2 \xi^2) \right),
		\end{equation*}
		with $\tau > 0$ the \emph{shape parameter} and $\sigma^2$ the \emph{scaling parameter}. We denote this situation by $Y \sim \mathcal{GL}(\sigma,\tau)$. Note that the variance of $Y$ is $\tau \sigma^2$. When $\tau = 1$, we recover the traditional Laplace law. The generalized Laplace laws are infinitely divisible \cite[Section 2.4.1]{Koltz2001laplace} and associated with the L\'evy triplet $(0,0, \nu_{\tau,\sigma^2})$ with \cite[Proposition 2.4.2]{Koltz2001laplace}
		
	\begin{equation*} \label{eq:levymeasurelaplace}
	\nu_{\tau,\sigma^2}(\drm x) = \tau \frac{1}{\lvert x \rvert} \mathrm{e}^{- 2 \lvert x \rvert / \sigma^2} \drm x. 
	\end{equation*}

		\begin{definition}
			We say that a L\'evy white noise $\dot{X}_{\mathrm{Laplace}}$ is a \emph{generalized Laplace white noise} if $\langle \dot{X}_{\mathrm{Laplace}} , \One_{[0,1]^d} \rangle \sim \mathcal{GL}(\sigma,\tau)$ for some $\tau, \sigma^2>0$. We call $\tau$ and $\sigma^2$ respectively the shape parameter and the scaling parameter of $\dot{X}_{\mathrm{Laplace}}$. When $\tau = 1$, we simply say that $\dot{X}_{\mathrm{Laplace}}$ is a \emph{Laplace white noise}.
		\end{definition}
		
	To the best of our knowledge, general integrability conditions for (generalized) Laplace noise has not been investigated in the literature. Proposition \ref{prop:laplacecase} provides such conditions. 
	
	\begin{proposition} \label{prop:laplacecase}
		For every generalized Laplace white noise $\dot{X}_{\mathrm{Laplace}}$, we have
		\begin{equation} \label{prop:domainlaplace}
			L^0(\dot{X}_{\mathrm{Laplace}}) = L^{\log, 2}(\R^d). 
		\end{equation}
		Moreover, for $0<p\leq 2$, we have
		\begin{equation} \label{prop:domainlaplacep}
			L^p(\dot{X}_{\mathrm{Laplace}}) = L^{p, 2}(\R^d). 
		\end{equation}		
	\end{proposition}
	
	\begin{proof}
		Let $0 \leq p \leq 2$. 
		We start from \eqref{eq:psimmetric} and write
		\begin{equation*}
		 \Psi_p(\xi)     =  \xi^2 \int_{|x| \leq 1/|\xi|} x^2 \nu_{\tau,\sigma^2} (\drm x) +  \lvert \xi \rvert^p  \int_{|x|> 1/ \lvert \xi \rvert} \lvert x \rvert^p \nu_{\tau,\sigma^2} (\drm x) := \Psi_{p,1}(\xi) + \Psi_{p,2}(\xi).
		\end{equation*}
		Without loss of generality, we   consider the case $\sigma^2 = 2$ and $\tau = 1$, in which case $\nu_{1,2}(\drm x) = \frac{\mathrm{e}^{-\lvert x \rvert}}{\lvert x \rvert} \drm x$. Then, by integration by parts, we have
		\begin{align*}
			\Psi_{p,1}(\xi) &= 2 \lvert \xi \rvert^2 \int_{0}^{1 / \lvert \xi\rvert} x \mathrm{e}^{- x} \drm x  = 2 \lvert \xi \rvert^2 \left( 1 - \mathrm{e}^{- 1 / \lvert \xi \rvert} (1 + \frac{1}{\lvert \xi \rvert} ) \right) 
		\end{align*}
		Hence, we have $\Psi_{p,1} (\xi) \underset{\xi \rightarrow \infty}{\longrightarrow} 2$ and $\Psi_{p,1} (\xi) \underset{\xi \rightarrow 0}{\sim} 2\lvert \xi\vert^2$. 
		
		For $\Psi_{p,2}(\xi) = \lvert \xi \rvert^p  \int_{|x|> 1/ \lvert \xi \rvert} \lvert x \rvert^p \nu_{\tau,\sigma^2} (\drm x)$, we shall distinguish between $p=0$ and $p>0$. For $p>0$, the function $x^{p-1}\mathrm{e}^{-x}$ is integrable over $\R$, so that $\Psi_{p,2}(\xi)  \underset{\xi \rightarrow \infty}{\sim} \left( \int_{\R} x^{p-1}\mathrm{e}^{-x} \drm x\right) \lvert \xi\vert^p$. For $p=0$, the function $x^{-1} \mathrm{e}^{-x} $ is not anymore integrable around $0$. Using the equivalence $ x^{-1} \mathrm{e}^{-x}  \underset{x \rightarrow 0}{\sim} x^{-1}$, we deduce that 
		$$ \Psi_{p,2}(\xi) = 2 \int_{   \frac{ 1}{\lvert \xi \rvert}}^{\infty}  x^{-1}\mathrm{e}^{-x} \drm x  \underset{\xi \rightarrow \infty}{\sim}   2 \int_{   \frac{ 1}{\lvert \xi \rvert}}^{1}  x^{-1}\mathrm{e}^{-x} \drm x\underset{\xi \rightarrow \infty}{\sim} 2 \int_{\frac{1}{\lvert \xi \rvert}}^1 x^{-1} \drm x = 2 \log \abs{\xi}.$$
		Moreover, since $p\leq 2$, we have, again by integration by parts,
		\begin{align*}
		\Psi_{p,2}(\xi) &= 2 \int_{x \lvert \xi \rvert>1} (x \lvert \xi \rvert)^p \mathrm{e}^{-x} \frac{ \drm x}{x} \leq 2 \int_{x \lvert \xi \rvert>1} (x \lvert \xi \rvert)^2 \mathrm{e}^{-x} \frac{ \drm x}{x}  = 2 \lvert \xi \rvert (1 + \lvert \xi \rvert) \mathrm{e}^{- 1 / \lvert \xi \rvert} ,
		\end{align*}
		implying that $\Psi_{p,2}(\xi) \underset{\xi \rightarrow 0}{=} o(\lvert \xi \rvert^2)$. By combining the results on $\Psi_{p,1}$ and $\Psi_{p,2}$, we obtain that
		\begin{itemize}
			\item for $0\leq p\leq2$, $\Psi_{p}(\xi) \underset{\xi \rightarrow 0}{\sim} 2 \lvert \xi \rvert^2$;
			\item for $0< p\leq 2$, $\Psi_{p}(\xi) \underset{\xi \rightarrow \infty}{\sim}  \left( \int_{\R} x^{p-1}\mathrm{e}^{-x} \drm x\right) \lvert \xi \rvert^p$;
			\item for $p=0$, $\Psi_{0}(\xi) = \Psi(\xi) \underset{\xi \rightarrow \infty}{\sim} 2 \log \lvert \xi\rvert$.
		\end{itemize}
		We apply now Proposition \ref{prop:criteria} to deduce \eqref{prop:domainlaplace} and \eqref{prop:domainlaplacep}.
	\end{proof}

		\subsubsection{Compound Poisson White Noises}\label{subsubsec:poisson}

\begin{definition}\label{def:poisson}
 A L\'evy white noise is a \emph{compound Poisson noise} if its L\'evy triplet has the form $(0,0,\nu)$ and if its L\'evy measure satisfies
 \begin{equation*}
 \lambda := \int_{\R} \nu(\drm x) < \infty.
 \end{equation*}
In that case, $\nu = \lambda \mathbb{P}$ with $\mathbb{P}$ a probability measure.  
\end{definition}	
	
One can represent a compound Poisson noise $\dot{X}_{\mathrm{Poisson}}$ as \cite[Theorem 1]{Unser2011stochastic}
		\begin{equation*}
			\dot{X}_{\mathrm{Poisson}} = \sum_{n\geq 0} a_n \delta ( \cdot - t_n)
		\end{equation*}
		with $\delta$ the Dirac distribution, $(a_n)$ i.i.d. random variables with probability law $\mathbb{P}$, and $(t_n)$ independent of $(a_n)$ such that, for every Borel set $B \in \R^d$, the random number of elements $t_n$ in $B$ follows a Poisson law with rate $\lambda \mathrm{Leb}(B)$.
	Then, $\dot{X}_{\mathrm{Poisson}}$ has a finite variance if and only if $\mathbb{P}$ has. In that case, it has a zero mean if and only if $\mathbb{P}$ has.
	
	The integration with respect to compound Poisson noise is treated for instance in \cite[Chapter 12]{Kallenberg2006foundations}. However, it relies on  $L^1$-type conditions, inherited from the fact that the stochastic integration follows the L\'evy-It\^o decomposition. We provide more general integrability conditions in Proposition \ref{prop:poissoncase}, that are not sensitive to $L^1$-type integrability. 
	
	\begin{proposition}\label{prop:poissoncase}
	If $\dot{X}_{\mathrm{Poisson}}$ is a symmetric compound Poisson white noise  with finite variance, then
	\begin{equation*}
				L^p (\dot{X}_{\mathrm{Poisson}}) = L^{p,2} (\R^d).
	\end{equation*}
for every $0 \leq p \leq 2$. 
	\end{proposition}
	
	\begin{proof}
		First, $L^p (\dot{X}_{\mathrm{Poisson}}) \subseteq L^{p,2} (\R^d)$ as for any symmetric L\'evy white noise, according to \eqref{eq:limitLp}.  Moreover, for a compound Poisson noise with finite variance, we have for every $q \in [0,2]$ that $ \int_{\R } \lvert x  \rvert^q   \mathbb{P}(\drm x) < \infty$. Therefore, we have
		\begin{align*}
			\Psi_p(\xi) &= \lambda \int_{\R } (  \lvert x \xi \rvert^p \wedge   \lvert x \xi \rvert^2) \mathbb{P}(\drm x)   \leq \lambda \min\left( \lvert \xi \rvert^p   \int_{\R } \lvert x  \rvert^p   \mathbb{P}(\drm x), \lvert \xi \rvert^2   \int_{\R } \lvert x  \rvert^2   \mathbb{P}(\drm x) \right)   \leq C ( \lvert \xi \rvert^p\wedge   \lvert \xi \rvert^2) = \rho_{p,2}(\xi), 
		\end{align*}	
		so that $\lVert f \rVert_{\Psi_p} \leq C \lVert f \rVert_{p,2}$. This means that $ L^{p,2} (\R^d) \subseteq L^p (\dot{X}_{\mathrm{Poisson}}) $, finishing the proof. 
	\end{proof}
	
	        We summarize the results of this section in Table \ref{table:domains}.
        
\begin{table*}[t!] 
\centering
\caption{Definition Domains of some L\'evy White Noises}

\begin{tabular}{ccccc} 
\hline
\hline\\[-2ex] 
White noise & Parameters  & $	\Phi_{\langle \dot{X}, \One_{[0,1]}\rangle}(\xi) $ &  $L^0(\dot{X})$ & $L^p(\dot{X})$ \\
& & & & $0<p \leq2$ 
\\
\hline\\[-1ex]
Gaussian & $\sigma^2 >  0$ & $\mathrm{e}^{- \sigma^2 \xi^2}$ & $L^2(\R^d)$ & $L^2(\R^d)$ \\[+1ex]
Pure drift & $\gamma \in \R$ & $\mathrm{e}^{\mathrm{i}\gamma \xi}$ & $L^1(\R^d)$ & $L^1(\R^d)$ \\[+1ex]
S$\alpha$S & $0<\alpha<2$ & $\mathrm{e}^{-\lvert \xi \rvert^\alpha}$ & $L^\alpha(\R^d)$ & $\begin{cases} L^\alpha(\R^d) &\mbox{if } p<\alpha \\ 
\{0\} & \mbox{if } p\geq \alpha \end{cases}$  \\[+1ex]
generalized & $\sigma^2 > 0$ & $\frac{1}{(1 + \sigma^2\xi/2)^\tau}$ & $L^{\log,2}(\R^d)$ & $L^{p,2}(\R^d)$ \\
Laplace & $\tau >0$  & & & \\[+1ex]
symmetric finite-variance & $\lambda >0$ & $\mathrm{e}^{\lambda (\widehat{\mathbb{P}} (\xi) -1 )}$ & $L^{0,2}(\R^d)$ & $L^{p,2}(\R^d)$ \\
compound Poisson & $\mathbb{P}$ & & & \\
\hline
\hline
\end{tabular} \label{table:domains}
\end{table*}

\section{Application to  SPDEs} \label{sec:application}

In this section, we consider stochastic partial differential equations (SPDEs) of the form
\begin{equation*} \label{eq:SDE}
	\Lop s  = \dot{X}
\end{equation*}	
with $\dot{X}$ a L\'evy white noise and $\Lop$ a differential operator.
{Considerable attention has been given in order to define various classes of random processes and fields depending on the Lévy noise model and the operator. A vast majority of works deal with the Gaussian case~\cite{oksendal2013stochastic,protter2005stochastic}. Generalizations include the analysis of SPDEs driven by stable white noise~\cite{pryhara2016stochastic,pryhara2017wave,Taqqu1994stable} or Lévy white noise~\cite{Applebaum2009levy}.}

Until now, we restricted ourselves to the study of L\'evy white noise. In this section, we   see how to apply our results to solve linear stochastic differential equations of the form
\begin{equation*} \label{eq:SDE}
	\Lop s  = \dot{X}
\end{equation*}	
with $\dot{X}$ a L\'evy white noise and $\Lop$ a differential operator.
We give new conditions of compatibility between the operator $\Lop$ and the white noise $\dot{X}$ such that the process $s$ exists as a generalized random process. 
Our results extend the results of previous works~\cite{Fageot2014,Unser2014sparse,Unser2014unifiedContinuous}.

\begin{theorem}\label{theo:generalcriterion}
	We consider a L\'evy white noise $\dot{X}$.
	We assume that $\Top$ is a continuous and linear operator from $\D(\R^d)$ to $L^0(\dot{X})$.
	Then, the mapping
	\begin{eqnarray}
	s :  & \D(\R^d) &\rightarrow L^0 (\Omega) \nonumber \\
			& \varphi  & \mapsto \langle s ,\varphi \rangle := \langle \dot{X}, \Top \{\varphi \}\rangle
	\end{eqnarray}
	specifies a generalized random process in the sense of Definition \ref{def:GRP}. 
	
	If moreover there exists an operator $\Lop$ such that $\Top \Lop^* \varphi = \varphi$ for every $\varphi \in \D(\R^d)$ (left-inverse property), then we have that
	\begin{equation}
		\Lop s = \dot{X}.
	\end{equation}
\end{theorem}

\begin{proof}
	The mapping $\varphi \mapsto \langle \dot{X}, \Top \{\varphi\} \rangle$ is well-defined for any  $\varphi \in \D(\R^d)$ because $\Top \{\varphi\} \in L^0(\dot{X})$ by assumption. 
	It is actually the composition of the operator $\Top$ with the random linear functional $\dot{X}$. These two mappings being linear, the composition is linear. Moreover, $\Top$ is continuous from $\D(\R^d)$ to $L^0(\dot{X})$ by assumption and $\dot{X}$ is continuous from $L^0(\dot{X})$ to $L^0(\Omega)$ according to Theorem \ref{prop:continuity}. Therefore, $s$ is  a linear and continuous mapping from $\D(\R^d)$ to $L
^0(\Omega)$ and therefore a valid generalized random process in $\D'(\R^d)$. 
	
	If now $\Top \Lop^*\{\varphi\} = \varphi$ for every $\varphi \in \D(\R^d)$, then the process $\Lop s$, defined as 
	\begin{equation}
		\Lop s \ : \varphi \mapsto \langle s, \Lop^* \varphi \rangle,
	\end{equation}
	satisfies the relation
	\begin{equation} \label{eq:whyonlyleftinverse}
		\langle \Lop s , \varphi \rangle = \langle \dot{X} , \Top \Lop^* \varphi \rangle = \langle \dot{X},\varphi\rangle 
	\end{equation}
	for every $\varphi \in \D(\R^d)$. Equivalently, we have shown that $\Lop s = \dot{X}$ as elements of $\D'(\R^d)$, as expected. 
\end{proof}

We justify shortly the assumptions of Theorem \ref{theo:generalcriterion}. 
Many differential operators admit a natural inverse $\Lop^{-1}$, that is typically defined using the Green's function of $\Lop$. 
A solution of \eqref{eq:SDE} can therefore be formally written as $s = \Lop^{-1} w$; that is, 
\begin{equation} \label{eq:formalrelation}
	\langle s, \varphi \rangle = \langle \Lop^{-1} \dot{X}, \varphi \rangle = \langle \dot{X}, (\Lop^*)^{-1} \{\varphi\} \rangle.
\end{equation}
In order to be valid, \eqref{eq:formalrelation} should at least be meaningful for any $\varphi \in \D(\R^d)$. It means in particular that $(\Lop^*)^{-1} \{\varphi\}$ should be in the domain of definition of the L\'evy white noise $\dot{X}$. 
However, for many differential operators, including the derivative, the natural inverse operator to $\Lop^*$ exists but is not stable in the sense that it is not continuous from $\D(\R^d)$ to any domain of definition $L^0(\dot{X})$. 
In that case, it is required to correct $(\Lop^*)^{-1}$ in order to make it stable.
The role of the corrected version of $(\Lop^*)^{-1}$ is played by $\Top$.  Thanks to \eqref{eq:whyonlyleftinverse}, we moreover see that we only require that $\Top$ is a \textit{left}-inverse of $\Lop^*$.
For the specification of stable left-inverses of important classes of differential operators, including the derivative of any order, fractional derivatives, fractional Laplacian, we refer the reader to \cite{Unser2014sparse}.

\begin{definition}\label{def:GLP}
A generalized random process constructed according to Theorem \ref{theo:generalcriterion} is called a \emph{generalized L\'evy process} (or \emph{generalized L\'evy field} when $d\geq 2$). The operator $\Lop$ is the \emph{whitening operator} of $s$ and $\dot{X}$ the underlying L\'evy white noise. 
\end{definition}

The following result links the stability properties of the corrected left-inverse operator $\Top$ with the finiteness of the generalized moments of the L\'evy measure of $\dot{X}$. 
 
 \begin{proposition} \label{prop:constructioncriterion}
 	We consider a symmetric L\'evy white noise without Gaussian  part $\dot{X}$ and a linear, continuous, and   shift-invariant operator $\Lop$. 
	We assume that, for $0\leq p_0, p_\infty \leq 2$, we have
	\begin{itemize}
		\item $m_{p_\infty, p_0}(\nu) = \int_{\R} \rho_{p_\infty, p_0} (t) \nu(\drm t) < \infty$, and
		\item the adjoint operator $\Lop^*$ admits a left-inverse $\Top$ that maps continuously $\D(\R^d)$ to $L^{p_0,p_\infty}(\R^d)$.
	\end{itemize}
	Then, there exists a generalized L\'evy process $s$ such that $\Lop s = \dot{X}$.
 \end{proposition}
 
 \begin{proof}
 	Applying \eqref{eq:limitmoreprecise} with $p=0$, the condition $\int_{\R} \rho_{p_\infty, p_0} (t) \nu(\drm t) < \infty$ ensures that $L^{p_0,p_\infty}(\R^d) \subset L^0(\dot{X})$. This embedding and the assumption on $\Top$ imply that $\Top$ maps continuously $\D(\R^d)$ to $L^0(\dot{X})$, and Theorem \ref{theo:generalcriterion} applies. 
 \end{proof}
 
We recall that the condition $m_{p_\infty,p_0}(\nu) < \infty$ introduced in \eqref{eq:genmoments} is connected with the local properties (parameter $p_0$) and the asymptotic properties (parameter $p_\infty$) of the L\'evy white noise.

 \paragraph{Comparison with previous works.}
 Theorem \ref{theo:generalcriterion} and Proposition \ref{prop:constructioncriterion} can be compared with other conditions of compatibility between the whitening operator $\Lop$ and the L\'evy white noise $\dot{X}$. The results are reformulated with our notation.

\begin{itemize}

	\item First of all, we differentiate between two types of solutions of the linear SPDE $\Lop s = w$. We say that $s$ is a \emph{generalized solution} if $\Lop s = w$ almost surely. In contrast, a generalized random process $s$ such that $\Lop s = w$ in law (that is, $\mathscr{P}_{\Lop s} = \mathscr{P}_w$) is called a \emph{solution in law}. {In recent works~\cite{Fageot2014,Unser2014sparse}, the solutions of \eqref{eq:SDE} were essentially constructed  relying on the Minlos-Bochner theorem.} One important contribution of this paper is to construct generalized solutions, what requires the identification of the domain of definition of the L\'evy white noise $\dot{X}$ to be as general as possible.	
	
	\item Throughout the paper, we have considered L\'evy white noise as random elements in $\D'(\R^d)$. This is in line with the original work of Gel'fand and Vilenkin \cite{GelVil4}.
It can be of interest, however, to restrict to the class of \emph{tempered} L\'evy white noise, that is, to consider L\'evy white noise, and by extension generalized L\'evy processes, as random elements in the space $\S'(\R^d)$.
The construction of generalized L\'evy processes in $\S'(\R^d)$ instead of $\D'(\R^d)$ can be found in \cite{Unser2014sparse}; see also \cite{bierme2017generalized} for a recent exposition of the main results in this framework. For a comparison between the two constructions, we refer to \cite[Chapter 3]{Fageot2014}.

Not every L\'evy white noise is tempered. Actually, a L\'evy white noise is tempered if and only if it has finite $p$th moments for some $p>0$ (see \cite{Dalang2015Levy}). 
The main point is that the construction is really analogous, since the spaces $\S(\R^d)$ and $\S'(\R^d)$ are nuclear.
Among the consequences, we mention that Theorem \ref{theo:generalcriterion} and Proposition \ref{prop:constructioncriterion} remain valid when replacing $\D(\R^d)$ by $\S(\R^d)$. In that case, the processes $\dot{X}$ and $s$ are both located in $\S'(\R^d)$. In what follows, when comparing these two results with previous contributions, we consider the version of the result for tempered generalized random processes.

	\item  For $1\leq p \leq 2$, the characteristic exponent $\psi$ is \emph{$p$-admissible} if $\abs{\psi(\xi)} + \abs{\xi} \abs{\psi'(\xi)} \leq C \abs{\xi}^p$. Note that the derivative $\psi'(\xi)$ is well-defined as soon as the first moment of the underlying infinitely divisible random variable is finite, what we assume now. This notion was introduced in \cite{Unser2014sparse} together with the following compatibility condition: if $\psi$ is $p$-admissible and $\Top$ continuously map $\S(\R^d)$ to $L^p(\R^d)$, then there exists a solution in law of $\Lop s = w$ with characteristic functional 
	\begin{equation} \label{eq:CFTop}
	\CF_s : \varphi \mapsto \CF_{\dot{X}} ( \Top \{\varphi\}).
	\end{equation}
	 A sufficient condition for the $p$-admissible is that $\int_{\R} \abs{t}^p \nu(\drm t) < \infty$. Therefore, \eqref{eq:CFTop} is a valid characteristic functional as soon as $\int_{\R} \abs{t}^p \nu(\drm t) < \infty$ and $\Top$ maps continuously $\S(\R^d)$ to $L^p(\R^d)$ for some $1 \leq p \leq 2$. 	
	We recover this result by selecting $p_0 = p_\infty = p$ in Proposition \ref{prop:constructioncriterion}. Actually, Proposition \ref{prop:constructioncriterion} extends this criterion in three ways. First, we can distinguish between the behavior of $\nu$ around $0$ and at $\infty$. Second, we do not restrict to the case $p \geq 1$ (this second improvement was already achieved in our work \cite{fageot2019scaling} thanks to a relaxation of the $p$-admissibility). Finally, as we have said already, we specify generalized solutions, and not only solutions in law, of $\Lop s = \dot{X}$.

	\item  {In the work \cite{Fageot2014} with A. Amini and M. Unser by the first author,} it has shown that the characteristic functional \eqref{eq:CFTop} specifies a generalized L\'evy process if $\int_{\R} \rho_{p_\infty,p_0} (t) \nu(\drm t)$ and $\Top$ maps continuously $\S(\R^d)$ to $L^{p_0,p_\infty}(\R^d)$ for $0 < p_\infty \leq p_0 \leq 2$ \cite[Theorem 5]{Fageot2014}. When $p_\infty \leq p_0$, we have that
	\begin{equation*}
		\max ( \abs{\xi}^{p_0} , \abs{\xi}^{p_\infty} ) \leq \rho_{p_0,p_\infty}(\xi) \leq  \abs{\xi}^{p_0} + \abs{\xi}^{p_\infty}.
	\end{equation*}
	Therefore, $L^{p_0,p_\infty}(\R^d) = L^{p_0}(\R^d) \cap L^{p_\infty}(\R^d)$ and we recover our previous result (at least for symmetric L\'evy white noise without Gaussian part).
	Moreover, Proposition \ref{prop:constructioncriterion} is an improvement, since one can consider $p_\infty > p_0$. In that case, $L^{p_0,p_\infty}(\R^d)$ contains but is strictly bigger than $L^{p_0}(\R^d) \cap L^{p_\infty}(\R^d)$ and the requirement on $\Top$ is less restrictive. With exactly the same idea, one could extend the class (generalized) L\'evy processes considered in \cite[Definition 5]{Fageot2016gaussian} by allowing cases when $p_\infty > p_0$, and therefore generalizing Theorem 5 in this paper. 

	\item Combining \eqref{eq:limitmoreprecise} and Proposition \ref{prop:continuityCF}, we generalize \cite[Theorem 2]{Amini2014sparsity} again   by considering the case $p_\infty > p_0$: we are able to specify a larger domain of definition and of continuity than $L^{p_0}(\R^d) \cap L^{p_\infty}(\R^d)$ in that case.
		
\end{itemize}

\appendix

\section{Generalized Orlicz Spaces} \label{app:orlicz}

\begin{definition} \label{def:phifun}
	We say that $\rho : \R \rightarrow \R^+$ is a \emph{$\varphi$-function} if $\rho(0) = 0$ and $\rho$ is symmetric, continuous, and nondecreasing on $\R^+$. The $\varphi$-function $\rho$ is  \emph{$\Delta_2$-regular} if 
	\begin{equation*}
		\rho( 2 \xi ) \leq M \rho(\xi)
	\end{equation*}
	for some $M , \xi_0>0$, and every $\xi \geq \xi_0$.
\end{definition}
	
\begin{definition} \label{def:orlicz}
	Let $\rho$ be a $\varphi$-function. For $f : \R^d \rightarrow \R$, we set 
	\begin{equation*}
	\rho(f) := \int_{\R^d} \rho(f (t)) \drm t.
	\end{equation*}
	The \emph{generalized Orlicz space} associated to $\rho$ is
	\begin{equation*}
		L^{\rho}(\R^d) := \left\{ f \text{ measurable}, \  \exists \lambda > 0, \ \rho(f / \lambda) < \infty  \right\}. 
	\end{equation*}
\end{definition}

Orlicz spaces were introduced in \cite{Birnbaum1931verallgemeinerung} as natural generalizations of $L^p$-spaces for $p\geq 1$. 
A systematic study with important extensions was done by J. Musielak \cite{Musielak2006orlicz}.
The initial theory deals with Banach spaces, excluding for instance the $L^p$-spaces with $0<p<1$. 
Definition \ref{def:orlicz} generalizes the Orlicz spaces in two ways: 
One does not require that $\rho$ is convex, neither that $\rho(\xi) \rightarrow \infty$ as $\xi \rightarrow \infty$.
The need for a non-locally convex framework (related to non-convex $\varphi$-function)   is notable in stochastic integration. It was initiated by K. Urbanik and W.A. Woyczyns \cite{Urbanik1967random}. 
It is at the heart of the study of the structure developed by Rajput and Rosinski. 
We follow here the exposition of M.M. Rao and Z.D. Ren in \cite[Chapter X]{Rao1991theory}. Proposition \ref{prop:orlicz} summarizes the results on generalized Orlicz spaces.

\begin{proposition} \label{prop:orlicz}
If $\rho$ is a $\Delta_2$-regular $\varphi$-function, then we have
\begin{equation*}
	L^{\rho}(\R^d) = \left\{ f \text{ measurable}, \  \forall \lambda > 0, \ \rho(f / \lambda) < \infty  \right\} = \left\{ f \text{ measurable}, \  \rho(f) < \infty  \right\}.
\end{equation*}
The space $L^{\rho}(\R^d)$ is a complete linear metric space for the F-norm 
\begin{equation*}
	\lVert f \rVert_{\rho} := \inf \{\lambda > 0 , \ \rho(f / \lambda) \leq \lambda \}
\end{equation*}	
on which simple functions are dense.
Moreover, we have the equivalence, for any sequence of elements $f_k \in L^{\rho}(\R^d)$,
\begin{equation*}
	\lVert f_k \rVert_{\rho} \underset{k \rightarrow \infty}{\longrightarrow} 0 \Leftrightarrow \rho(f_k) \underset{k \rightarrow \infty}{\longrightarrow} 0.
\end{equation*}
\end{proposition}

\section*{Acknowledgements}
The authors are warmly grateful to Prof. Robert Dalang and Prof. Michael Unser for fruitful discussions. This research was partially supported by the Swiss National Foundation for Scientific Research and the European Research Council under Grant H2020-ERC (ERC grant agreement No 692726 - GlobalBioIm).

{\footnotesize
\bibliographystyle{plain}
\bibliography{references}

\begin{thebibliography}{10}

\bibitem{abdesselam2018towards}
A.~Abdesselam.
\newblock Towards three-dimensional conformal probability.
\newblock {\em p-Adic Numbers, Ultrametric Analysis and Applications},
  10(4):233--252, 2018.

\bibitem{Abdesselam2020second}
A.~Abdesselam.
\newblock A second-quantized kolmogorov--chentsov theorem via the operator
  product expansion.
\newblock {\em Communications in Mathematical Physics}, pages 1--54, 2020.

\bibitem{Adler1983representations}
R.J. Adler, D.~Monrad, R.H. Scissors, and R.~Wilson.
\newblock Representations, decompositions and sample function continuity of
  random fields with independent increments.
\newblock {\em Stochastic Process. Appl.}, 15(1):3--30, 1983.

\bibitem{Amini2014sparsity}
A.~Amini and M.~Unser.
\newblock Sparsity and infinite divisibility.
\newblock {\em {IEEE} Transactions on Information Theory}, 60(4):2346--2358,
  2014.

\bibitem{Applebaum2009levy}
D.~Applebaum.
\newblock {\em L{\'e}vy {P}rocesses and {S}tochastic {C}alculus}.
\newblock Cambridge {U}niversity {P}ress, 2009.

\bibitem{aziznejad2018wavelet}
S.~Aziznejad, J.~Fageot, and M.~Unser.
\newblock Wavelet analysis of the besov regularity of {L}\'evy white noises.
\newblock {\em arXiv preprint arXiv:1801.09245}, 2018.

\bibitem{berger2019levydriven2}
D.~Berger.
\newblock L\'evy driven linear and semilinear stochastic partial differential
  equations.
\newblock {\em arXiv preprint arXiv:1907.01926}, 2019.

\bibitem{berger2019levydriven}
D.~Berger.
\newblock L{\'e}vy driven {CARMA} generalized processes and stochastic partial
  differential equations.
\newblock {\em Stochastic Processes and their Applications},
  130(10):5865--5887, 2020.

\bibitem{Bertoin1998levy}
J.~Bertoin.
\newblock {\em L{\'e}vy Processes}, volume 121.
\newblock Cambridge University Press, 1998.

\bibitem{bierme2017generalized}
H.~Bierm{\'e}, O.~Durieu, and Y.~Wang.
\newblock Generalized random fields and {L}\'{e}vy's continuity theorem on the
  space of tempered distributions.
\newblock {\em Commun. Stoch. Anal.}, 12(4):Article 4, 427--445, 2018.

\bibitem{billingsley2013convergence}
P.~Billingsley.
\newblock {\em Convergence of probability measures}.
\newblock John Wiley \& Sons, 2013.

\bibitem{Birnbaum1931verallgemeinerung}
Z.~Birnbaum and W.~Orlicz.
\newblock {\"U}ber die {V}erallgemeinerung des {B}egriffes der zueinander
  konjugierten {P}otenzen.
\newblock {\em Studia Mathematica}, 3(1):1--67, 1931.

\bibitem{Bogachev2007measure}
V.I. Bogachev.
\newblock {\em Measure theory. {V}ol. {I}, {II}}.
\newblock Springer-Verlag, Berlin, 2007.

\bibitem{Bostan2013map}
E.~Bostan, J.~Fageot, U.S. Kamilov, and M.~Unser.
\newblock {MAP} estimators for self-similar sparse stochastic models.
\newblock In {\em Proceedings of the Tenth International Workshop on Sampling
  Theory and Applications (SampTA’13), Bremen, Germany}, pages 197--199,
  2013.

\bibitem{Bostan.etal2013}
E.~Bostan, U.S. Kamilov, M.~Nilchian, and M.~Unser.
\newblock Sparse stochastic processes and discretization of linear inverse
  problems.
\newblock {\em {IEEE} Transactions on Image Processing}, 22(7):2699--2710,
  2013.

\bibitem{Bottcher2014levy}
B.~B{\"o}ttcher, R.L. Schilling, and J.~Wang.
\newblock {\em L{\'e}vy {M}atters III: L{\'e}vy-{T}ype {P}rocesses:
  {C}onstruction, {A}pproximation and {S}ample {P}ath {P}roperties}, volume
  2099.
\newblock Springer, 2014.

\bibitem{Brockwell2010carma}
P.J. Brockwell and J.~Hannig.
\newblock {CARMA} $(p,q)$ generalized random processes.
\newblock {\em Journal of Statistical Planning and Inference},
  140(12):3613--3618, 2010.

\bibitem{Camia2015planar}
F.~Camia, C.~Garban, and C.M. Newman.
\newblock Planar {I}sing magnetization field i. uniqueness of the critical
  scaling limit.
\newblock {\em The Annals of Probability}, 43(2):528--571, 2015.

\bibitem{Cartier1963processus}
P.~Cartier.
\newblock Processus al{\'e}atoires g{\'e}n{\'e}ralis{\'e}s.
\newblock {\em S{\'e}minaire Bourbaki}, 8:425--434, 1963.

\bibitem{Chelkak2012conformal}
D.~Chelkak, C.~Hongler, and K.~Izyurov.
\newblock Conformal invariance of spin correlations in the planar {I}sing
  model.
\newblock {\em arXiv preprint arXiv:1202.2838}, 2012.

\bibitem{chung}
K.L. Chung.
\newblock {\em A course in probability theory}.
\newblock Academic Press [A subsidiary of Harcourt Brace Jovanovich,
  Publishers], New York-London, second edition, 1974.
\newblock Probability and Mathematical Statistics, Vol. 21.

\bibitem{clarkson2016characteristic}
E.~Clarkson and H.H. Barrett.
\newblock Characteristic functionals in imaging and image-quality assessment:
  tutorial.
\newblock {\em JOSA A}, 33(8):1464--1475, 2016.

\bibitem{Dalang2015Levy}
R.C. Dalang and T.~Humeau.
\newblock L{\'e}vy processes and {L}{\'e}vy white noise as tempered
  distributions.
\newblock {\em The Annals of Probability}, 45(6b):4389--4418, 2017.

\bibitem{dalang2019random}
R.C. Dalang and T.~Humeau.
\newblock Random field solutions to linear {SPDE}s driven by symmetric pure
  jump {L}{\'e}vy space-time white noises.
\newblock {\em Electronic Journal of Probability}, 24, 2019.

\bibitem{Dalang1992}
R.C. Dalang and J.B. Walsh.
\newblock The sharp {M}arkov property of {L}{\'e}vy sheets.
\newblock {\em The Annals of Probability}, pages 591--626, 1992.

\bibitem{Daubechies1988orthonormal}
I.~Daubechies.
\newblock Orthonormal bases of compactly supported wavelets.
\newblock {\em Communications on {P}ure and {A}pplied {M}athematics},
  41(7):909--996, 1988.

\bibitem{Fageot2014}
J.~Fageot, A.~Amini, and M.~Unser.
\newblock On the continuity of characteristic functionals and sparse stochastic
  modeling.
\newblock {\em Journal of Fourier Analysis and Applications}, 20:1179--1211,
  2014.

\bibitem{Fageot2015wavelet}
J.~Fageot, E.~Bostan, and M.~Unser.
\newblock Wavelet statistics of sparse and self-similar images.
\newblock {\em SIAM Journal on Imaging Sciences}, 8(4):2951--2975, 2015.

\bibitem{Fageot2017multidimensional}
J.~Fageot, A.~Fallah, and M.~Unser.
\newblock Multidimensional {L}{\'e}vy white noise in weighted {B}esov spaces.
\newblock {\em Stochastic Processes and Their Applications}, 127(5):1599--1621,
  2017.

\bibitem{Fageot2016gaussian}
J.~Fageot, V.~Uhlmann, and M.~Unser.
\newblock {G}aussian and sparse processes are limits of generalized {P}oisson
  processes.
\newblock {\em Applied and Computational Harmonic Analysis}, in press.

\bibitem{fageot2019scaling}
J.~Fageot and M.~Unser.
\newblock Scaling limits of solutions of linear stochastic differential
  equations driven by {L}{\'e}vy white noises.
\newblock {\em Journal of Theoretical Probability}, 32(3):1166--1189, 2019.

\bibitem{Fageot2017besov}
J.~Fageot, M.~Unser, and J.P. Ward.
\newblock On the {B}esov regularity of periodic {L}{\'e}vy noises.
\newblock {\em Applied and Computational Harmonic Analysis}, 42(1):21 -- 36,
  2017.

\bibitem{Fageot2017nterm}
J.~Fageot, M.~Unser, and J.P. Ward.
\newblock The $n$-term approximation of periodic generalized {L}{\'e}vy
  processes.
\newblock {\em Journal of Theoretical Probability}, in press.

\bibitem{Fernique1967lois}
X.~Fernique.
\newblock Lois ind{\'e}finiment divisibles sur l'espace des distributions.
\newblock {\em Inventiones mathematicae}, 3(4):282--292, 1967.

\bibitem{Furlan2017tightness}
M.~Furlan and J.C. Mourrat.
\newblock A tightness criterion for random fields, with application to the
  ising model.
\newblock {\em Electronic Journal of Probability}, 22, 2017.

\bibitem{Gelfand1955generalized}
I.M. Gel'fand.
\newblock Generalized random processes.
\newblock {\em Doklady Akademii Nauk SSSR}, 100:853--856, 1955.

\bibitem{GelVil4}
I.M. Gel'fand and N.Y. Vilenkin.
\newblock {\em Generalized {F}unctions. {V}ol. 4: {A}pplications of {H}armonic
  {A}nalysis}.
\newblock Academic Press, New York-London, 1964.

\bibitem{Hida1980brownian}
T.~Hida.
\newblock Brownian motion.
\newblock In {\em Brownian Motion}, pages 44--113. Springer, 1980.

\bibitem{Hida2004}
T.~Hida and Si-Si.
\newblock {\em An {I}nnovation {A}pproach to {R}andom {F}ields}.
\newblock World Scientific Publishing Co., Inc., River Edge, NJ, 2004.

\bibitem{hummel2019stochastic}
F.~Hummel.
\newblock {\em Stochastic Transmission and Boundary Value Problems}.
\newblock PhD thesis, Universit{\"a}t Konstanz, 2019.

\bibitem{hummel2020sample}
F.~Hummel.
\newblock Sample paths of white noise in spaces with dominating mixed
  smoothness.
\newblock {\em arXiv preprint arXiv:2005.10858}, 2020.

\bibitem{Ito1954distributions}
K.~It{\^o}.
\newblock Stationary random distributions.
\newblock {\em Kyoto Journal of Mathematics}, 28(3):209--223, 1954.

\bibitem{Ito1984foundations}
K.~It{\^o}.
\newblock {\em Foundations of {S}tochastic {D}ifferential {E}quations in
  {I}nfinite {D}imensional {S}paces}, volume~47.
\newblock SIAM, 1984.

\bibitem{Kallenberg2006foundations}
O.~Kallenberg.
\newblock {\em Foundations of modern probability}.
\newblock Springer Science \& Business Media, 2006.

\bibitem{Koltz2001laplace}
S.~Koltz, T.J. Kozubowski, and K.~Podgorski.
\newblock {\em The Laplace Distribution and Generalizations}.
\newblock Boston, MA: Birkhauser, 2001.

\bibitem{Lee2006levy}
Y.-J. Lee and H.-H. Shih.
\newblock L{\'e}vy white noise measures on infinite-dimensional spaces:
  Existence and characterization of the measurable support.
\newblock {\em Journal of Functional Analysis}, 237(2):617--633, 2006.

\bibitem{Lokka2004stochastic}
A.~L{\o}kka, B.~{\O}ksendal, and F.~Proske.
\newblock Stochastic partial differential equations driven by {L}{\'e}vy
  space-time white noise.
\newblock {\em The Annals of Applied Probability}, 14(3):1506--1528, 2004.

\bibitem{Musielak2006orlicz}
J.~Musielak.
\newblock {\em Orlicz spaces and modular spaces}, volume 1034 of {\em Lecture
  Notes in Mathematics}.
\newblock Springer-Verlag, Berlin, 1983.

\bibitem{Oksendal2004white}
G.~Di Nunno, B.~{\O}ksendal, and F.~Proske.
\newblock White noise analysis for {L}{\'e}vy processes.
\newblock {\em Journal of Functional Analysis}, 206(1):109--148, 2004.

\bibitem{oksendal2013stochastic}
B.~Oksendal.
\newblock {\em Stochastic differential equations: an introduction with
  applications}.
\newblock Springer Science \& Business Media, 2013.

\bibitem{protter2005stochastic}
P.E. Protter.
\newblock Stochastic differential equations.
\newblock In {\em Stochastic integration and differential equations}, pages
  249--361. Springer, 2005.

\bibitem{Pruitt1981growth}
W.E. Pruitt.
\newblock The growth of random walks and {L}{\'e}vy processes.
\newblock {\em The Annals of Probability}, 9(6):948--956, 1981.

\bibitem{pryhara2016stochastic}
L.I. Pryhara and G.M. Shevchenko.
\newblock Stochastic wave equation in a plane driven by spatial stable noise.
\newblock {\em Modern Stochastics: Theory and Applications}, 3(3):237--248,
  2016.

\bibitem{pryhara2017wave}
L.I. Pryhara and G.M. Shevchenko.
\newblock Wave equation with stable noise.
\newblock {\em Theory of Probability and Mathematical Statistics}, 96:145--157,
  2017.

\bibitem{Rajput1989spectral}
B.S. Rajput and J.~Rosinski.
\newblock Spectral representations of infinitely divisible processes.
\newblock {\em Probability Theory and Related Fields}, 82(3):451--487, 1989.

\bibitem{Rao1991theory}
M.M. Rao and Z.D. Ren.
\newblock {\em Theory of {O}rlicz spaces}, volume 146 of {\em Monographs and
  Textbooks in Pure and Applied Mathematics}.
\newblock Marcel Dekker, Inc., New York, 1991.

\bibitem{Taqqu1994stable}
G.~Samorodnitsky and M.S. Taqqu.
\newblock {\em Stable {N}on-{Gaussian} {P}rocesses: {S}tochastic {M}odels with
  {I}nfinite {V}ariance}.
\newblock Stochastic Modeling. Chapman \& Hall, New York, 1994.

\bibitem{Sato1994levy}
K.~Sato.
\newblock {\em L\'evy {P}rocesses and {I}nfinitely {D}ivisible
  {D}istributions}, volume~68.
\newblock Cambridge University Press, Cambridge, 2013.

\bibitem{Schilling1997Feller}
R.L. Schilling.
\newblock On {F}eller processes with sample paths in {B}esov spaces.
\newblock {\em Mathematische Annalen}, 309(4):663--675, 1997.

\bibitem{Schwartz1966distributions}
L.~Schwartz.
\newblock {\em Th\'eorie des distributions}.
\newblock Hermann, 1966.

\bibitem{Simon1979functional}
B.~Simon.
\newblock {\em Functional integration and quantum physics}, volume~86.
\newblock Academic press, 1979.

\bibitem{Treves1967}
F.~Tr{\`e}ves.
\newblock {\em Topological {V}ector {S}paces, {D}istributions and {K}ernels}.
\newblock Academic Press, New York-London, 1967.

\bibitem{Unser2011stochastic}
M.~Unser and P.~D. Tafti.
\newblock Stochastic models for sparse and piecewise-smooth signals.
\newblock {\em {IEEE} Transactions on Signal Processing}, 59(3):989--1006,
  2011.

\bibitem{Unser2014sparse}
M.~Unser and P.~D. Tafti.
\newblock {\em An Introduction to Sparse Stochastic Processes}.
\newblock Cambridge University Press, 2014.

\bibitem{Unser2014unifiedContinuous}
M.~Unser, P.~D. Tafti, and Q.~Sun.
\newblock A unified formulation of {G}aussian versus sparse stochastic
  processes---{P}art {I}: {C}ontinuous-domain theory.
\newblock {\em Information Theory {IEEE} Transactions}, 60(3):1945--1962, 2014.

\bibitem{Urbanik1967random}
K.~Urbanik and W.A. Woyczynski.
\newblock A random integral and {O}rlicz spaces.
\newblock {\em Bulletin de l'acad\'emie polonaise des sciences - {S}{\'e}rie
  des sciences math\'ematiques, astronomiques et physiques}, 15(3):161, 1967.

\bibitem{ProbaBanach1987}
N.~Vakhania, V.~Tarieladze, and S.~Chobanyan.
\newblock {\em Probability distributions on {B}anach spaces}, volume~14 of {\em
  Mathematics and its Applications (Soviet Series)}.
\newblock D. Reidel Publishing Co., Dordrecht, 1987.

\bibitem{Walsh1986introduction}
J.B. Walsh.
\newblock An introduction to stochastic partial differential equations.
\newblock In {\em {\'E}cole d'{\'E}t{\'e} de Probabilit{\'e}s de Saint Flour
  XIV-1984}, pages 265--439. Springer, 1986.

\bibitem{zygmund}
R.L. Wheeden and A.~Zygmund.
\newblock {\em Measure and integral}.
\newblock Marcel Dekker, Inc., New York-Basel, 1977.
\newblock An introduction to real analysis, Pure and Applied Mathematics, Vol.
  43.

\end{thebibliography}
}

\end{document}